\documentclass[11pt]{article}

\usepackage{amsmath,amsfonts,amssymb,MnSymbol}

\newcommand{\R}{{\mathbb R}}

\usepackage{bbm}
\def\B{{\cal B}}
\def\E{{\cal E}}
\def\F{{\cal F}}

\def\H{{\cal H}}

\def\F{{\cal F}}
\def\V{{\cal V}}
\def\A{{\cal A}}
\def\e{{\cal E}}

\def\ci{\mbox{$C_0^\infty(E)$}}
\def\tt{\mbox{$(T_t)_{t > 0}$}}
\def\ttb{\mbox{$(\overline{T}_t)_{t > 0}$}}

\def\ga{\mbox{$(G_\alpha)_{\alpha > 0}$}}

\def\e{{\varepsilon}}

\newtheorem{defn}{{\it DEFINITION}}
\newtheorem{theo}[defn]{{\it THEOREM}}
\newtheorem{lem}[defn]{{\it LEMMA}}
\newtheorem{prop}[defn]{{\it PROPOSITION}}
\newtheorem{cor}[defn]{{\it COROLLARY}}
\newtheorem{rem}[defn]{{\it REMARK}}
\newtheorem{exam}[defn]{{\it EXAMPLE}}
\newenvironment{proof}{{\bf\it Proof }}{{\vskip 0.1cm \hfill$\Box$}}

\begin{document}
\noindent
{\Large\bf Recurrence criteria for generalized Dirichlet forms}

\bigskip
\noindent
{\bf Minjung Gim},
{\bf Gerald Trutnau}
\\

\noindent
{\small{\bf Abstract.} We develop sufficient analytic conditions for recurrence and transience of non-sectorial 
perturbations of possibly non-symmetric Dirichlet forms on a general state space. These form an important 
subclass of generalized Dirichlet forms which were introduced in \cite{St1}. In case there exists an associated process, 
we show how the analytic conditions imply recurrence and transience in the classical probabilistic sense. As an application,
we consider a generalized Dirichlet form given on a closed or open subset of $\mathbb{R}^d$ which is given as a divergence free first order perturbation 
of a non-symmetric energy form. Then using volume growth conditions of the sectorial and non-sectorial first order part, we derive an explicit criterion 
for recurrence. Moreover, we present concrete examples with applications to Muckenhoupt weights and counterexamples. The counterexamples 
show that the non-sectorial case differs qualitatively from the symmetric or non-symmetric sectorial case. Namely, we make the observation that 
one of the main criteria for recurrence in these cases fails to be true for generalized Dirichlet forms.} \\



\section{Introduction}\label{1}
Recurrence and transience of Markov processes, as well as related problems are important topics in probability theory. These were hence studied by many authors under various probabilistic and analytic aspects in discrete and in continuous time (see for instance 
\cite{Ba,BeCiR,FOT,F07,Ge,H,K,WaMaU,MT,No,O,O13,Pi,Sc,Stu1} and references therein).\\
Here, we take a rather analytic point of view which fits to the frame of possibly unbounded and discontinuous coefficients. The main purpose of this paper is to develop recurrence and transience criteria for (Markov processes $\mathbb{M}$ corresponding to) a generalized Dirichlet form which can be expressed as a linear perturbation of a sectorial Dirichlet form. More precisely, we consider a nice Hausdorff space $E$, 
a $\sigma$-finite measure $\mu$ on the Borel $\sigma$-algebra $\B(E)$ of $E$ and a generalized Dirichlet form $\E$ that can be decomposed as
\begin{equation}\label{in}
\E(u,v)=\E^0(u,v)+\int_E u Nv d\mu,
\end{equation}
where $(\E^0,D(\E^0))$ is a sectorial Dirichlet form on $L^2(E,\mu)$ that is dominated by $\E$ on a subspace of the diagonal of $D(\E^0)$ and $(N,D(N))$ is a linear operator on $L^2(E,\mu)$. The precise conditions are formulated in (H1)-(H3) of Section 2 below.\\
Here as a warning, we emphasize that we use the term \lq\lq sectorial\rq\rq\ exclusively in the sense of strong sector condition (cf. Remark \ref{titr}(a) and end of Remark 5(b)).\\
The class of generalized Dirichlet forms as in (\ref{in}) is quite large. It contains symmetric Dirichlet forms as in \cite{FOT}, sectorial Dirichlet forms as in \cite{MR} (and also \cite{O13}, if the dual semigroup is supposed to be sub-Markovian there) and time-dependent Dirichlet forms as in \cite{O04}. After having introduced the basic notions, for even more general forms as in ($\ref{in}$),  namely generalized Dirichlet forms satisfying (H1)-(H2), we derive some domination principle on the diagonal (see Theorem \ref{tit} and Remark \ref{titr}) and the existence of a nice reference function in case of transience (see Lemma \ref{gg}). Our main result for general forms as in (\ref{in}) is Theorem \ref{notran1} and its Corollary \ref{res} which constitute a generalization of the symmetric case of \cite{FOT} and of the sectorial case of \cite{O13}, if $(\widehat{T}_t)_{t>0}$ is sub-Markovian there (cf. Remark \ref{res2}).\\
Recurrence and transience are described through potential operators and the potential operators can be defined in an analytic way through an underlying $C_0$-semigroup of contractions as for instance in (\ref{dg}) below or in a probabilistic way where the potential operator is defined through an underlying Markov process $\mathbb{M}$ as at the beginning of Subsection \ref{2.2}. In Subsection \ref{2.2}, we follow the main lines of the well-known work \cite{Ge} to point out the connection of the analytic recurrence and transience to the probabilistic one. In particular, if the generalized Dirichlet form in (\ref{in}) is associated to a right process $\mathbb{M}$ as at the beginning of Subsection \ref{2.2}, i.e. if
$$
G_\alpha f=E.\; \Big [ \int_0^\infty e^{-\alpha t} f(X_t)dt \Big ] \ \mu\text{-a.e.}
$$
for any bounded $f\in L^2(E,\mu)$ and $\alpha>0$, then the analytic recurrence (resp. transience) of $\E$ can be described probabilistically as in Proposition \ref{rpro} (resp. Proposition \ref{tpro}). Moreover, if the transition function $(p_t)_{t> 0}$ of $\mathbb{M}$ is strong Feller, then the $\mu$-a.e. statements of Propositions \ref{tpro} and \ref{rpro} can be transformed into everywhere statements as explained at the end of Subsection \ref{2.2}. Thus we obtain pointwise recurrence as in the case of (H\"older) continuous or locally bounded coefficients (cf. for instance \cite{Ba}, \cite{Pi}) even though in our situation the coefficients may be discontinuous and unbounded. In general only the transition from $\mu$-a.e. to $\E$-quasi-everywhere statements is possible in Propositions \ref{tpro} and \ref{rpro} through standard Dirichlet form theory arguments.\\
As an application in Section \ref{3}, we consider an open or closed subset $E$ of $\R^d$ and adapting the arguments of \cite{St1} in particular to the case with reflection (cf. Lemma \ref{prop} and its proof in Section \ref{4}), we construct a generalized Dirichlet form on $L^2(E,\mu)$, $d\mu=\varphi dx$, $\varphi>0$ $dx$-a.e., that extends
\begin{equation}\label{efg}
\E(f,g)=\int_E \langle A\nabla f, \nabla g \rangle d\mu -\int_E \langle B,\nabla f\rangle gd\mu,
\end{equation}
where $A=(a_{ij})_{1\leq i,j \leq d}$ is a possibly non-symmetric matrix of locally $\mu$-integrable functions and $B:=(B_1,\ldots ,B_d)\in L^2_{loc}(E,\R^d,\mu)$ is $\mu$-divergence free (see (\ref{div1}) below). For the precise conditions, we refer to Section \ref{3}. In particular, we show that the form (\ref{efg}) fits into the frame of (\ref{in}) and we obtain first sufficient recurrence and transience criteria for (\ref{efg}) by applying the results of Subsection \ref{2.1} (cf. Corollary \ref{cor} and Remark \ref{cor2}). Then following a construction scheme of \cite{FOT} that we adapt to the non-sectorial case (cf. Lemmas \ref{con1} and \ref{con2}), we show that recurrence of $\E$ in (\ref{efg}) implies recurrence of its symmetric part (cf. Theorem \ref{rec3}) and conservativeness of $\E$ (cf. Corollary \ref{cons3}). For ease of exposition some proofs of Section \ref{3} are postponed to Section \ref{4}.\\
In Subsection \ref{3.1}, we derive explicit conditions for recurrence under the existence of a nice function $\rho$ (see beginning of Subsection \ref{3.1}) which always exists if $E$ is closed and so in particular if  $E=\R^d$. Our main result here is Theorem \ref{ee} that characterizes recurrence in terms of volume growth. It can be seen as a generalization of \cite[Theorem 3]{Stu1} in the Euclidean case. 
\\
In Subsection \ref{3.2} we present examples and counterexamples. The counterexamples show that the non-sectorial case differs from the symmetric and from the non-symmetric sectorial case. In order to explain the difference, we first recall the well-known sufficient conditions for recurrence in the sectorial case:\\
If $(\E^0,D(\E^0))$ is a symmetric Dirichlet form on $L^2(E,\mu)$, then the existence of $(\chi_n)_{n \geq 1}\subset D(\E^0)$ such that $0\leq \chi_n \leq 1$, $\lim_{n \rightarrow \infty}\chi_n =1$ $\mu$-a.e. and $\lim_{n \rightarrow \infty} \E^0(\chi_n,\chi_n)=0$ is an equivalent condition for (analytic) recurrence of $(\E^0,D(\E^0))$ (see \cite[Theorem 1.6.3]{FOT} and beginning of Subsection \ref{3.1}). In addition, if $(\E^0,D(\E^0))$ is a sectorial Dirichlet form and strictly irreducible, then the existence of  $(\chi_n)_{n \geq 1}\subset D(\E^0)$ such that $0\leq \chi_n \leq 1$, $\lim_{n \rightarrow \infty}\chi_n =1$ $\mu$-a.e. and $\lim_{n \rightarrow \infty} \E^0(\chi_n,\chi_n)=0$ is a sufficient condition for recurrence of $(\E^0,D(\E^0))$ (\cite[Theorem 1.3.9]{O13}). In Subsections \ref{exam1}, \ref{exam2}, we present several counterexamples of generalized Dirichlet forms as in (\ref{efg}) for which there exists $(\chi_n)_{n\geq 1}$ as above with
$$
\lim_{n \rightarrow \infty} \E(\chi_n,\chi_n)=0,
$$  
but $\E$ is not recurrent.  
In Subsection \ref{exam3} we discuss concrete examples in the case where the density $\varphi$ is in some Muckenhoupt class.\\
Section \ref{4} is as already mentioned devoted to the postponed proofs of Section \ref{3}.\\
\section{Analytic and probabilistic characterization of recurrence and transience}\label{2}
This section consists of two parts. In the first part, we characterize recurrence and transience analytically in the non-sectorial case and derive an analytic criterion for a generalized Dirichlet form to be recurrent or more generally non-transient.\\
In the second part, we show that the analytic characterization of recurrence and transience indeed implies recurrence and transience in the classical probabilistic sense in case there exists a process associated with the generalized Dirichlet form.
\subsection{Framework and a general criterion for recurrence and transience of a generalized Dirichlet form}\label{2.1}
Let $E$ be a Hausdorff topological space such that its Borel $\sigma$-algebra $\B(E)$ is generated by the set of all continuous functions on $E$ and let $\mu$ be a $\sigma$-finite positive measure on $\B(E)$. For $1\leq p<\infty$, let $L^p(E,\mu)$ be the space of equivalence classes of $p$-integrable functions with respect to $\mu$ and $L^\infty(E,\mu)$ be the space of $\mu$-essentially bounded functions. Let further $\H:=L^2(E,\mu)$ with inner product $(\ ,\ )$. The support of a function $u$ on $E$ (=support of $|u|d\mu$) is denoted by supp($u)$. For any set of functions $W$ on $E$, we will denote by $W_0$ the set of functions $u\in W$ which have a compact support in $E$ and by $W_b$ the set of functions in $W$ which are bounded $\mu$-a.e. and let $W_{0,b}:=W_0 \cap W_b$. Let $(\A,\V)$ be a Dirichlet form (not necessarily symmetric) on $\H$ in the sense of \cite[I. Definition 4.5]{MR}. So $\V$ is a real Hilbert space with respect to the norm $\|u\|^2_\V:=\A(u,u)+(u,u)$. Denote the dual space of $\V$ by $\V'$. Assume that there exists a linear operator $(\Lambda,D(\Lambda,\H))$ on $\H$ satisfying the following assumptions:\\
\centerline{}
\begin{itemize}
	\item[(H1)] $(\Lambda,D(\Lambda,\H))$ generates a sub-Markovian $C_0$-semigroup of contractions $(U_t)_{t>0}$ on $\H$ that can be restricted to a $C_0$-semigroup on $\V$.
\end{itemize}
\centerline{}
It follows by (H1) that the conditions (D1) and (D2) in \cite[Chapter I]{St2} are satisfied. So $\Lambda : D(\Lambda,\H)\cap \V \longrightarrow \V'$ is closable. Let $(\Lambda,\F)$ be the closure of $\Lambda : D(\Lambda,\H)\cap \V \longrightarrow \V'$. Then $\F$ is a real Hilbert space with corresponding norm
$$
\|u\|^2_\F:=\|u\|^2_\V+\|\Lambda u \|^2_{\V'}.
$$
Let $\E$ be the bilinear form associated with $(\A,\V$) and $(\Lambda,D(\Lambda,\H))$ (see \cite[I. Definition 2.9]{St2}). Then $\E$ is a generalized Dirichlet form (see \cite[I. Proposition 4.7]{St2}). In particular, for $u\in \F,\ v\in \V$, $\E$ can be written as
$$
\E(u,v)=\ \A(u,v)-{}_{\V'}\langle \Lambda u, v \rangle_\V.
$$
Let ($G_\alpha)_{\alpha>0}$ and ($\widehat{G}_\alpha)_{\alpha>0}$ on $\H$ be associated with $\E$, i.e. ($G_\alpha)_{\alpha>0}$ is the sub-Markovian $C_0$-resolvent of contractions on $\H$ satisfying $G_\alpha (\H)\subset \F$,
$$
\E_\alpha (G_\alpha f, g)=(f,g),\quad f\in \H,\ g\in \V,
$$
where $\E_\alpha(u,v):=\E(u,v)+\alpha(u,v)$ for $\alpha>0$ and $(\widehat{G}_\alpha)_{\alpha>0}$ is the adjoint $C_0$-resolvent of contractions of $\ga$ (see \cite[I. Proposition 3.6]{St2}). By \cite[I. Proposition 1.5]{MR}, there exists exactly one linear operator ($L,D(L))$ on $\H$ corresponding to $(G_\alpha)_{\alpha>0}$. Let $\tt$ and $(\widehat{T}_t)_{t>0}$ be $C_0$-semigroup of contractions corresponding to $\ga$ and $(\widehat{G}_\alpha)_{\alpha>0}$ respectively. Next, we assume\\
\begin{itemize}
	\item[(H2)]$(\widehat{T_t})_{t> 0} $ is sub-Markovian.\\
\end{itemize}
Then $(T_t)_{t> 0}$ restricted to ${L^1(E,\mu) \cap L^2(E,\mu)}$  can be extended to a semigroup of contractions on $L^1(E,\mu)$, which is actually equivalent to (H2). Since $(T_t)_{t> 0}$ is positivity preserving, so is its $L^1(E,\mu)$-version. Let $f \in L^1(E,\mu)$ with $f \geq 0 $. Then for $0\leq N\leq M$,
$$
0\leq \int_0^N T_t f dt \leq \int_0^M T_t f dt
$$
and for $0\leq \beta \leq \alpha$,
$$
0\leq \int_0^\infty e^{-\alpha t} T_t f dt\leq \int_0^\infty e^{-\beta t} T_t f dt.
$$
Hence
\begin{equation}\label{dg}
G f:=\lim_{N\rightarrow \infty} \int_0^N T_t f dt =\lim_{\alpha \rightarrow 0} \int_0 ^\infty e^{-\alpha t} T_t f dt \leq \infty
\end{equation}
is uniquely defined $\mu$-a.e. $G$ is called potential operator associated with $\tt$.\\
We do assume (H1) and (H2) throughout the whole Section \ref{2}.
\centerline{}
\centerline{}
\begin{defn}\label{rec1}$ $
\begin{itemize}
	\item[(a)] $\tt$ is said to be recurrent, if for any $f\in L^1(E,\mu)$ with $f\geq 0$ $\mu$-a.e., we have
\[
G f=0 \ \text{ or } \  \infty \ \mu\text{-a.e.}
	\]
	\item[(b)]   $\tt$ is said to be transient, if there exists $g\in L^1(E,\mu)$ with $g>0$ $\mu$-a.e. such that
\[
Gg<\infty \ \mu\text{-a.e.}
\]
	\item[(c)]  Likewise, we can define recurrence and transience of any operator which is a generator of a positivity preserving semigroup of contractions on $L^1(E,\mu)$.
\centerline{}
\end{itemize}
\end{defn}
\centerline{}
\begin{defn}$ $
\begin{itemize}
\item[(a)] A measurable set $B\in \B(E)$ is called weakly invariant set relative to $\tt$, if\\
\centerline{}
\centerline{ $\displaystyle T_t (f \cdot 1_B ) (x)=0$ for $\mu$-a.e. $x\in E\setminus B$}
\centerline{}
for any $t> 0,$ $f\in L^2(E,\mu)$.
\item[(b)] $\tt$  is said to be strictly irreducible, if for any weakly invariant set $B$ relative to \tt, we have $\mu(B)=0$ or $\mu(E\setminus B)=0$.
\end{itemize}
\end{defn}

\centerline{}
\begin{rem}\label{rec2} From \cite[Section 2]{K}, we deduce:
\begin{itemize}
\item[(a)] \tt $ $ is transient, if and only if $Gf<\infty$ $\mu$-a.e. for any $f\in L^1(E,\mu)$ with $f\geq 0$ $\mu$-a.e.
\item[(b)] If $g\in L^1(E,\mu)$ with $g>0$ $\mu$-a.e., then $\{ x\in E\,:\,  Gg(x)=\infty\}$ is a weakly invariant set relative to \tt. Consequently, if $\tt$ is strictly irreducible, then it is either transient or recurrent.
\item[(c)] If for some $t>0$, there exists a $\mu\otimes\mu$-a.e. strictly positive measurable function $(p_t(x,y))_{x,y\in E}$ with
$$
T_t f(x) = \int_E p_t(x,y) f(y) \mu(dy)
$$
for $\mu$-a.e. $x\in E$ and any $f\in L^2(E,\mu)$, then $\tt$ is strictly irreducible.
\item[(d)] In the symmetric case (cf. \cite{FOT}), $B\in \B(E)$ is weakly invariant, if and only if it is invariant in the sense of \cite[Chapter 1.1.6]{FOT}. Therefore, a symmetric Dirichlet form is irreducible if and only if it is strictly irreducible.
\end{itemize}
\end{rem}
\centerline{}
Now, we shall show that the transience of $\tt$ is determined by the symmetric part of the corresponding generalized Dirichlet form under some domination on the diagonal.\\
\begin{theo}\label{tit}
If there exists a sectorial Dirichlet form $(\E^0,D(\E^0))$ in the sense of \cite[I. Definition 4.5 and I. (2.4)]{MR} such that $D(L)_b \subset D(\E^0)$ and $\E^0(u,u)\leq \E(u,u)$ for any $u\in D(L)_b$, then the transience of  $(\E^0,D(\E^0))$ implies the transience of $\tt$.
\end{theo}
\begin{proof} 
If we assume $(\E^0,D(\E^0))$ is transient, then by Remark \ref{rec2}(a), there exists $g^0 \in L^1(E,\mu)_b$ with $g^0>0$ $\mu$-a.e. such that $G^0 g^0 <\infty$ $\mu$-a.e. where $G^0$ is the potential operator associated with $(\E^0,D(\E^0))$. Let $g:=g^0 / \max(G^0 g^0,1)$. Then $g\in L^1(E,\mu)_b$ and $\int_E g G^0 g d\mu <\infty$. According to \cite[Theorem 1.3.9]{O13}, there exists a constant $K_g>0$ depending on $g$ and the sector constant of ($\E^0,D(\E^0))$ such that
\begin{equation}\label{leq1}
\int_E |u|g d\mu \leq K_g \sqrt{\E^0 (u,u)}
\end{equation}
for any $u\in D(\E^0)$. Choosing $u$ to be $G_\alpha g \in D(L)_b$ in (\ref{leq1}) where $\alpha>0$, we get
$$
\int_E g G_\alpha g d\mu \leq K_g \sqrt{\E^0(G_\alpha g,G_\alpha g)} \leq K_g \sqrt{ \E(G_\alpha g,G_\alpha g)} \leq K_g \sqrt{\int_E g G_\alpha g d\mu}.
$$
Therefore $\int_E g G_\alpha g d\mu \leq K^2_g$ for any $\alpha>0$, and it follows by B. Levi's theorem that
$$
\int_E g Gg d\mu \leq K^2_g.
$$
Consequently, $G g<\infty$ $\mu$-a.e.
\end{proof}
\centerline{}
\begin{rem}\label{titr}
\begin{itemize}
\item[(a)] In this article the term "sectorial" is exclusively meant in the sense of satisfying condition (2.4) of \cite[Chapter I. 2]{MR}.
\item[(b)] Let $(\E^0,D(\E^0))$ be a sectorial Dirichlet form on $\H$. Then its symmetric part $(\widetilde{\E}^0,D(\E^0))$ is a symmetric Dirichlet form on $\H$ $($see \cite[I. Exercise 4.6]{MR}$)$. By Theorem \ref{tit}, we obtain: a sectorial Dirichlet form $(\E^0,D(\E^0))$ is transient, if and only if $(\widetilde{\E}^0,D(\E^0))$ is transient. Note that if $(\E^0,D(\E^0))$ is only assumed to satisfy the weak sector condition as in  
\cite[(2.3) of Chapter I]{MR}, then its symmetric part may be recurrent, while $(\E^0,D(\E^0))$ is not. This can be seen in Example \ref{exam22} below.
\end{itemize}
\end{rem}
\centerline{}
Since \tt $ $ is a sub-Markovian $C_0$-semigroup of contractions on $ L^2(E,\mu)$, it can be defined on $L^\infty(E,\mu)$. Let $f\in L^\infty(E,\mu)$, with $f\geq 0$ $\mu$-a.e. and choose an increasing sequence of non-negative functions $(f_n)_{n \geq 1} \subset L^2 (E,\mu)_b$ converging to $f$ $\mu$-a.e. Then for any $t> 0$,
$$
T_t f:= \lim_{n \rightarrow \infty} T_t f_n
$$
exists $\mu$-a.e. independently of choice of functions $(f_n)_{n \geq 1}$. Furthermore for any $t,s > 0$,\\
\centerline{}
\centerline{$T_t T_s f =T_t (\lim_{n \rightarrow \infty} T_s f_n)= \lim_{n \rightarrow \infty}T_t T_s f_n =\lim_{n \rightarrow \infty} T_{t+s}f_n =T_{t+s}f, \ \mu$-a.e.}\\
\centerline{}
Consequently, \tt$ $ can be considered as a sub-Markovian semigroup of contractions on $L^\infty(E,\mu)$. The potential operator $G$ relative to \tt$ $ can be regarded as an operator on $L^\infty(E,\mu)$.\\
Using an idea from \cite[Theorem 15]{Sc} about invariant sets of discrete semigroups in the proof of the next lemma, we show that $g$ and $Gg$ in Definition \ref{rec1}(b) can be chosen $\mu$-uniformly bounded.
\centerline{}
\begin{lem}\label{gg}
If $(T_t)_{t> 0}$ is transient, then there exists a function $g\in L^1(E,\mu)_b$ with $g>0$ $\mu$-a.e. and $G g\in L^\infty(E, \mu)$.
\end{lem}
\begin{proof}
Fix a function $f\in L^1(E,\mu)_b$ with $f>0$ $\mu$-a.e., $\|f\|_{L^\infty(\mu)} \leq 1$ and $\|f\|_{L^1(\mu)} \leq 1$. By Remark \ref{rec2}(a), we have\\
$$
0<Gf<\infty\ \mu\text{-a.e.}
$$
Define functions for $m,k \geq 1$, by
$$
g_{mk}:=Gf \wedge m -T_k(Gf \wedge m)
$$
where $a\wedge b :=\min\{a,b\}$. Then $\|g_{mk} \|_{L^\infty(\mu)} \leq m$. Moreover, if $x\in E$ is such that $Gf(x)<m$, then since $T_k$ is positivity preserving,
\begin{align*}
g_{mk}(x)&=Gf(x)-T_k(Gf \wedge m)(x) \\
&\geq Gf(x)-T_k(Gf)(x) =\int_0^k T_t f(x) dt \geq 0.
\end{align*}
If $x\in E$ is such that $Gf(x)\geq m$, then since $T_k$ is sub-Markovian, we have
$$
g_{mk}(x)= m-T_k(Gf \wedge m)(x) \geq 0.
$$
Consequently, $g_{mk}\geq 0$ $\mu$-a.e. Define for $m,k \geq 1$,
$$
A_m:=\{ x\in E:\  (m-1) \leq Gf(x) < m\}
$$
and
$$
B_k:=\Big \{x\in E:\ \int_0^{k-1} T_t f(x) dt=0\ \text{but} \ \int_0^{k} T_t f(x) dt>0\Big \}.
$$
Since $(T_t)_{t> 0}$ is transient, $\cup_{m=1}^\infty A_m=\cup_{k=1}^\infty B_k=E$ up to some $\mu$-negligible set. Without loss of generality, we may assume that $\mu(A_m \cap B_k)<\infty$ for any $m,k\geq 1$. Otherwise, we may subdivide $A_m\cap B_k$ in countably many pairwise disjoint sets with finite $\mu$-measure and proceed as below.  Let $c_{mk}:=\max(1, \mu(A_m \cap B_k))$ and
$$
\widetilde{g}_{mk}:= g_{mk} \cdot 1_{A_m\cap B_k}.
$$
Then $\| \widetilde{g}_{mk} \|_{L^\infty(\mu)} \leq m$, $\| \widetilde{g}_{mk} \|_{L^1(\mu)} \leq m\cdot \mu(A_m \cap B_k)$, $\widetilde{g}_{mk}>0$ on $A_m \cap B_k$ and
\begin{align*}
G\widetilde{g}_{mk} \leq G g_{mk} &= G \Big( Gf\wedge m -T_k(Gf\wedge m) \Big) =\int_0^k T_t(Gf\wedge m) dt \leq mk
\end{align*}  $\mu$-a.e. on $E$.
Consequently,
$$
g:=\sum_{m=1}^\infty \sum_{k=1}^\infty \frac{\widetilde{g}_{mk}}{2^m 2^k c_{mk}}
$$
satisfies the desired properties.
\end{proof}
\centerline{}
\centerline{}
Next, we give a general criterion for recurrence in case the generalized Dirichlet form can be represented by a linear perturbation of a sectorial Dirichlet form. By this, we mean that there exist a sectorial Dirichlet form $(\E^0,D(\E^0))$ in the sense of \cite[I. Definition 4.5 and I. (2.4)]{MR} with $D(L)_b \subset D(\E^0)$ and a linear operator $(N,D(N))$ on $\H$ such that
\begin{equation}\label{euv}
\E(u,v)=(-Lu,v)=\E^0(u,v)+\int_E u Nv d\mu
\end{equation}
for any $u\in D(L)_b$ and $v\in D(N)\cap D(\E^0)$. Here the linear operator $(N,D(N))$ needs not to be a generator of a $C_0$-semigroup of contractions on $\H$.
Thus from now on, we assume that the generalized Dirichlet form $\E$ satisfies the following condition:\\
\begin{itemize}
\item[(H3)] $\E$ can be decomposed as in (\ref{euv}) and
$$
\E^0(u,u)\leq \E(u,u)
$$
for any $u\in D(L)_b$.
\end{itemize}
\centerline{}
Let
$$
D:= \{u \in D(N)\cap D(\E^0) :\  Nu \in L^1(E,\mu)\}.
$$
\centerline{}
For the given sectorial Dirichlet form $(\E^0,D(\E^0))$, we define the $extended$ $Dirichlet$ $space$ of $D(\E^0)$ as the set of all measurable functions $u$ with $|u|<\infty$ $\mu$-a.e. for which there exists an $\E^0$-Cauchy sequence $(u_n)_{n\geq 1}\subset D(\E^0)$ such that
$$
\lim_{n\rightarrow \infty}u_n=u\ \ \mu\text{-a.e.} 
$$
(see \cite[Chapter 1.3]{O13}). Since the Dirichlet form $(\E^0,D(\E^0))$ is sectorial, for $u$ in the extended Dirichlet space, 
$$
\E^0(u,u):=\lim_{n\rightarrow \infty}\E^0(u_n,u_n)
$$ 
exists and is independent of the choice of $(u_n)_{n\geq 1}\subset D(\E^0)$ (this can be shown as in the paragraph right before \cite[Theorem 1.3.9]{O13}).
\begin{theo}\label{notran1}
Suppose $\tt$ is transient and let $g$ be as in Lemma \ref{gg}. Then $Gg$ is in the extended Dirichlet space of $D(\E^0)$ and one can define $\E^0(Gg,u):= \lim_{\alpha\rightarrow 0} \E^0(G_\alpha g,u)$ for $u\in D(\E^0)$. Moreover, if $u\in D$, then
\begin{equation}\label{equation2}
(u,g) = \E^0(Gg, u) + \int_E Gg\cdot Nu  d\mu.
\end{equation}
\end{theo}
\begin{proof}
Suppose that  $\tt$ is transient and let $g\in L^1(E,\mu)_b$ with $g>0$ $\mu$-a.e. such that $G g\in L^\infty(E, \mu)$. Since for every $\alpha >0$,
$$
\E^0(G_\alpha g,G_\alpha g)\leq \E(G_\alpha g,G_\alpha g)= -\int_E LG_\alpha g G_\alpha gd\mu\leq \int_E gGg d\mu< \infty,
$$
there exists an $\E^0$-Cauchy sequence $(g_n)_{n\geq 1}\subset D(\E^0)$ consisting of a Ces\`aro mean of $(G_{\alpha_n}g)_{n\geq 1}$ for $\alpha_n \to 0$ as $n\to \infty$. Indeed, this follows from the theorems of Banach/Alaoglu and Banach/Saks applied in the abstract completion of $(\E^0,D(\E^0))$. Consequently, $\E^0(Gg, Gg)=\lim_{n \rightarrow \infty}\E^0(g_n,g_n)$ exists and $\lim_{n\rightarrow \infty} g_n=Gg$ $\mu$-a.e. On the other hand, by the special form of $(g_n)_{n \geq 1}$, $\E^0(Gg,Gg)\leq\int_E gGg d\mu$ and $\lim_{n\rightarrow \infty} \E^0(g_n,u)=\lim_{n \rightarrow \infty} \E^0(G_{\alpha_n}g,u)$ for any $u\in D(\E^0)$. By our assumption for $u\in D$, we have for any $n\geq 1$
$$
(g,u)-\alpha_n(G_{\alpha_n} g,u)=(-LG_{\alpha_n} g, u)= \E^0(G_{\alpha_n} g,u)+\int_E  G_{\alpha_n} g\cdot Nu d\mu.
$$
Since $\lim_{n\rightarrow \infty} G_{\alpha_n} g=Gg$ $\mu$-a.e. and $Nu \in L^1(E,\mu)$, we obtain by Lebesgue's theorem
$$
\lim_{n\rightarrow \infty} \int_E G_{\alpha_n} g \cdot Nu d\mu=\int_E  Gg\cdot Nu d\mu.
$$
Since  $\E_\alpha (G_\alpha g,G_\alpha g)= \int_E gG_\alpha g d\mu\leq \int_E gGgd\mu$, for any $\alpha_n >0$, we get
$$
|\alpha_n (G_{\alpha_n}g,u)|\leq \sqrt{\alpha_n \int_E gGg d\mu}\cdot \| u\|_{L^2(\mu)}.
$$
Let $n \to \infty$ (thus $\alpha_n\to 0$) and we obtain (\ref{equation2}).
\end{proof}\\
\centerline{}
\begin{cor}\label{res}$ $
\begin{itemize}
\item[(a)] If there exists a sequence of functions $(\chi_n)_{n\geq 1} \subset D$ with $0\leq \chi_n \leq 1$, $\lim_{n\rightarrow \infty} \chi_n =1$ $\mu$-a.e. satisfying
\begin{equation}\label{suf}
\lim_{n \rightarrow \infty} \Big( \E^0(g,\chi_n) + \int_E g N\chi_n d\mu\Big) =0,
\end{equation}
for any non-negative bounded $g$ in the extended Dirichlet space of $D(\E^0)$, then $\tt$ is not transient.
\item[(b)] If $\tt$ is strictly irreducible and there is a sequence of functions $(\chi_n)_{n\geq 1} \subset D$ with $0\leq \chi_n \leq 1$, $\lim_{n\rightarrow \infty} \chi_n =1$ $\mu$-a.e. satisfying (\ref{suf}), then $\tt$ is recurrent by Remark \ref{rec2}(b).
\end{itemize}
\end{cor}
\centerline{}
\centerline{}
\begin{rem}\label{res2}
If $\E$ is a symmetric Dirichlet form, then we can drop the assumption that $\tt$ is strictly irreducible in 
Corollary \ref{res}(b). Indeed in this case one can use the $($weak$)$ invariance of $E_d:=\{x\in E:\ Gg(x)<\infty\}$, $g\in L^1(E,\mu)$, $g>0$ $\mu$-a.e. and the reduced form on $E_d$ in order to conclude $($cf. proof of \cite[Theorem 1.6.3]{FOT}$)$. Thus Corollary \ref{res}(b) can be seen as generalization of the symmetric case \cite[Theorem 1.6.5]{FOT}. However, in our general non-symmetric situation, even if $N\equiv 0$, $E_d$ is not weakly invariant in general and tools as in the symmetric case are not at hand. Consequently, strict irreducibility is imposed in Corollary \ref{res}(b). Moreover, Theorem \ref{notran1} is a generalization \cite[Theorem 1.3.9]{O13} in case $(\widehat{T}_t)_{t>0}$ is sub-Markovian.
\end{rem}

\subsection{Connection to recurrence and transience in the classical sense}\label{2.2}
For all notations, results that may not be defined, proved and cited in this Subsection, we refer to \cite{FOT}.\\
\centerline{}
Let $\mathbb{M}=(\Omega,(\F_t)_{t\geq 0},(X_t)_{t\geq 0},(P_x)_{x\in E_\triangle})$ with life time $\zeta$  be a right process with state space $E$, resolvent $R_\alpha f(x):=E_x \Big[ \int_0^\infty e^{-\alpha t} f(X_t)dt \Big ]$ and semigroup $p_t f(x):= E_x \Big [f(X_t) \Big ],$ $x\in E$, $\alpha >0,$ $t>0$, $f\in B(E)_b$, where $B(E)$ denotes the set of Borel measurable functions on $E$, $E_x$ denotes the expectation with respect to $P_x$. We assume that the measure $\mu$ is excessive relative to $(p_t)_{t>0}$, i.e.
$$
\int_E p_t 1_A(x) \mu(dx) \leq \mu(A),\ A\in \B(E).
$$
Hence, $(p_t)_{t>0}$ can be regarded as a linear operator sending a $\mu$-equivalence class to another $\mu$-equivalence class and can be extended as a linear operator on $L^1(E,\mu)$.\\
\centerline{}
We are able to define recurrence and transience of $\mathbb{M}$ as in Definition \ref{rec1}. The Markov process $\mathbb{M}$ is said to be recurrent, if for any $f\in L^1(E,\mu)$ with $f\geq 0$ $\mu$-a.e., we have
$$
E_x \Big [ \int_0^\infty f(X_t)dt \Big ] =0\ \text{or}\ \infty \text{ for } \mu \text{-a.e.}\ x\in E.
$$
$\mathbb{M}$ is said to be transient, if there exists $g\in L^1(E,\mu)$ with $g>0$ $\mu$-a.e., such that\\
\centerline{}
\centerline{$\displaystyle E_x \Big [ \int_0^\infty g(X_t)dt \Big ]<\infty$ \text{for} $\mu$-a.e. $x\in E$.}
\centerline{}
 A $\mu$-measurable set $B$ is said to be weakly invariant relative to $\mathbb{M}$, if for any $t>0$,\\
\centerline{}
\centerline{$\displaystyle E_x \Big [ 1_B(X_t) \Big ]=0$ for $\mu$-a.e. $x\in E\setminus B$.}
\centerline{}
$\mathbb{M}$ is said to be strictly irreducible, if for any weakly invariant set $B$ relative to $\mathbb{M}$, we have $\mu(B)=0$ or $\mu(E\setminus B)=0.$ A function $u$ is said to be excessive, if $p_t u(x) \nearrow u(x)$ as $t \searrow 0$ for all $x\in E$. Here, according to \cite[Theorem A.2.5]{FOT}, we will make the assumption that excessive functions are nearly Borel measurable with respect to $\mathbb{M}$. In particular, we may assume that sets like $\{u>0\}$, etc., are $\mu$-measurable.\\
For $\omega \in \Omega$, define the first hitting time $\sigma_B$ and last exit time $L_B$ from $B$ by
$$
\sigma_B(\omega):=\inf \{  t>0:\  X_t(\omega)\in B\}\ \text{ and } \ L_B(\omega):=\sup \{t\geq 0 :\ X_t(\omega)\in B\}.
$$
Note that $\sigma_B$ is $\F_t$-stopping time and $L_B$ is $\F_\infty$-measurable. Let
$$
p_B(x):= P_x(\sigma_B<\infty).
$$
Now we can characterize recurrence and transience of $\mathbb{M}$ in terms of its sample paths behavior following \cite{Ge}. More precisely, we have the following:\\
\begin{prop}\label{tpro}
$\mathbb{M}$ is transient, if and only if there exists a sequence of Borel finely open sets $(B_n)_{n\geq 1}$ increasing to $E$ up to some $\mu$-negligible set and for any $n\geq 1$\\
\begin{equation}\label{lbn}
P_x(L_{B_n}<\infty)=1\ \text{for}\ \mu\text{-a.e.}\ x\in E.
\end{equation}
\end{prop}
\centerline{}
\begin{proof}
For $g\in L^1(E,\mu)$, $g\geq0$ pointwise, define 
$$
Rg(x):=E_x \Big [ \int_0^\infty g(X_t)dt \Big ]\in [0,\infty], \  x\in E.
$$
Then $Rg=Gg$ $\mu$-a.e. Assume that $\mathbb{M}$ is transient. Then by definition there exists $g\in L^1(E,\mu)$, $g>0$ $\mu$-a.e. such that $E_x \Big [ \int_0^\infty g(X_t)dt \Big ]<\infty$ for $\mu$-a.e. $x\in E$. Modifying $g$ on a $\mu$-negligible set, we may assume that $g\geq0$ pointwise, $g>0$ $\mu$-a.e. and $Rg<\infty$ $\mu$-a.e. In particular, $Rg$ is excessive, hence finely continuous and so
$$
B_n:=\{x\in E:\  Rg(x)>\frac{1}{n}\}, \ n\geq 1,
$$
are finely open sets that increase to  $E$ up to some $\mu$-negligible set. Let $(\theta_t)_{t\geq 0}$ be the shift operator of $\mathbb{M}$. Since $Rg<\infty$ $\mu$-a.e., $((Rg(X_t))_{t\ge 0}, (\F_t)_{t\geq 0}, P_x)$ is a positive supermartingale for $\mu$-a.e. $x\in E$. We hence obtain by the optional sampling theorem for positive supermartingales, which holds for arbitrary $(\F_t)$-stopping times, that for $n\geq 1$ and $t>0$
\begin{align*}
p_tRg(x) &= E_x[Rg(X_{t}) ] \\
&\geq E_x[Rg(X_{t+\sigma_{B_n}\circ\theta_t})]\\
&\geq E_x[Rg(X_{t+\sigma_{B_n}\circ\theta_t})\,;\, t+\sigma_{B_n}\circ\theta_t<\infty]\\
&\geq \frac{1}{n}\,P_x(t+\sigma_{B_n}\circ\theta_t<\infty).
\end{align*}
The last inequality followed since $X_{t+\sigma_{B_n}\circ\theta_t}$ is in the closure of $B_n$ by the right-continuity of the sample paths and so by the fine continuity of $Rg$, we have $Rg(X_{t+\sigma_{B_n}\circ\theta_t})\ge \frac{1}{n}$. Thus $P_x(L_{B_n}<\infty)=1$ for $\mu$-a.e $x\in E$. \\
Conversely, suppose there exists a sequence of Borel finely open sets $(B_n)_{n\geq 1}$ increasing to $E$ up to some $\mu$-negligible set satisfying (\ref{lbn}). Let
$$
g_{B_n}(x):= P_x(L_{B_n}>0)=P_x(\sigma_{B_n}<\infty).
$$
Then $g_{B_n}$ is excessive and bounded. 
Set
$$
g_{nk}(x):=\frac{1}{k}\Big (g_{B_n}(x)-p_{k}g_{B_n}(x)\Big ).
$$
Using (\ref{lbn}) and a similar argument to \cite[proof of (3.1) Lemma on p. 404 and proof of (2.2) Proposition (iii')$\Rightarrow$(i) on page 402]{Ge}, we construct
$$
\widetilde{g}:= \sum_{n=1}^\infty \sum_{k=1}^\infty \frac{g_{nk}}{2^n 2^k} >0 \ \mu\text{-a.e.}
$$
with $\widetilde{g}\geq0$ pointwise such that $R\widetilde{g}<\infty$ $\mu$-a.e. Finally, since $\mu$ is $\sigma$-finite, there exists $h >0$ $\mu$-a.e., $h\geq0$ pointwise with $h\in L^1(E,\mu)$. Then $g:=\widetilde{g}\wedge h$ is in $L^1(E,\mu)$ and $Rg<\infty$ $\mu$-a.e. Therefore $\mathbb{M}$ is transient.
\end{proof}\\
\centerline{}
A set $B\subset E$ is called $\mu$-polar, if there exists $B_0\in \B(E)$, $B_0\supset B$, such that 
$$
\int_E P_x(\sigma_{B_0}<\infty) \mu(dx)=0.
$$
\begin{prop}\label{rpro}
Let $\mathbb{M}$ be strictly irreducible and recurrent. Then the following holds:
\begin{itemize}
	\item[(a)] Any bounded excessive function $u$ satisfies for any $t>0$,\\
\centerline{}
\centerline{$\displaystyle p_t u(x)=u(x)$ for $\mu$-a.e. $x\in E$.}
	\item[(b)] Any excessive function is constant on $E$ $\mu$-a.e.
	\item[(c)] If  there are two finely open sets $G, \tilde G\subset E$, with $G\cap \tilde G=\emptyset$ and $\mu(G), \mu(\tilde G)>0$, then 
$\displaystyle P_x(\zeta=\infty)=1$ for $\mu$-a.e. $x\in E$.
\item[(d)]  If $B$ is not $\mu$-polar and finely open in $E$, then\\
\centerline{}
\centerline{$\displaystyle P_x(L_B<\infty)=0$ for $\mu$-a.e. $x\in E$.}
\end{itemize}
\end{prop}
\begin{proof}
(a) Let $u$ be a bounded and excessive function. Then $t\to p_tu$ is decreasing as $t\to \infty$. Set $\psi(x):= \lim_{t\rightarrow \infty}p_t u(x)$. Then for any $s > 0$,
$$
p_s\psi(x)=\displaystyle \lim_{t\rightarrow \infty} p_{s+t}u(x)=\psi(x).
$$
Set $g:=u-\psi$. Then for any $t>0$,
$$
p_tg=p_tu-p_t\psi=p_tu-\psi
$$
and $p_tg \nearrow g $ as $t\searrow 0$, since $u$ is excessive. It follows that $g$ is also excessive and bounded. Furthermore, since $p_t g(x) \to 0$ as $t\to \infty$ and $p_t g(x) \to g(x)$ as $t\to 0$,
$$
g_n:=n(g-p_{1/n}g)
$$satisfies $Rg_n \nearrow g$ as $n \to \infty$. If $\mu(\{ x\in E\, : \, g_n(x)>0\})>0$, then there exists $\varepsilon >0$ with $\mu(\{  x\in E\,:\, g_n(x)>\varepsilon \})>0$. Set $A:=\{ x\in E\,:\, g_n(x)>\varepsilon\}$. Then $\varepsilon 1_A < g_n$. Since $\mu$ is $\sigma$-finite, we may assume that $\mu(A)<\infty$. Thus $\{ x\in E\,:\, R 1_A(x)=\infty\}$ is weakly invariant and so by strict irreducibility and recurrence  of $\mathbb{M}$, $R 1_A=\infty$ $\mu$-a.e., hence $R g_n=\infty$ $\mu$-a.e. However, since $g$ is bounded, we must have that $Rg_n$ is bounded for any $n\ge 1$.  Thus $g_n=0$ $\mu$-a.e. for any $n \geq 1$, which further implies that $g=0$ $\mu$-a.e. Equivalently, $u=\lim_{t\rightarrow \infty}p_tu$. Since $t\to p_tu$ is decreasing and $p_tu\leq u$, we obtain for all $t>0$,\\
\centerline{}
\centerline{$\displaystyle p_tu(x)=u(x)$ for $\mu$-a.e. $x\in E$.}
\centerline{}
(b) Let $f$ be a non-constant excessive function. Then there exist $0<a\le b$ with $A:=\{x\in E:\ f(x)<a\}$ and $B:=\{x\in E:\  f(x)>b\}$ satisfying\\
\centerline{}
\centerline{$\mu(A)>0$ and  $\mu(B)>0$.}
\centerline{}
Let $C:=\{x\in E:\  p_B(x)=1\}$. Then for any $x\in B$, we have $P_x(\sigma_B=0)=1$, hence $B\subset C$ and so $\mu(C)>0$. Set $D:=\{x\in E:\  p_B(x)<1\}$. Since $p_B$ is bounded and excessive, $D$ is nearly Borel measurable  and for any $t>0$, by (a)\\
\centerline{}
\centerline{$\displaystyle 1=p_B(x)=p_t p_B(x)=E_x[p_B(X_t)], \quad$ for $\mu$-a.e. $x\in C$.}
\centerline{}
Consequently, for each $t>0$, $P_x(X_t \in D)=0$ for $\mu$-a.e $x\in C$, hence $R 1_D=0$ $\mu$-a.e. on $C$. Since $\mu$ is $\sigma$-finite, there exists a $\mu$-integrable function $h$, with $0<h<1$ $\mu$-a.e. If $\mu(D)>0$, then $\mu(\{ x\in E\, : \, R (h 1_D)(x)>0\})>0$ and so as in (a) by strict irreducibility, we obtain $R (h1_D)=\infty$ $\mu$-a.e. which further implies that $R 1_D=\infty$ $\mu$-a.e. But this is a contradiction, since $\mu(C)>0$. Thus $\mu(D)=0$.  Now, for $\mu$-a.e.  $x\in A$, we have since  $((f(X_t))_{t\ge 0}, (\F_t)_{t\geq 0}, P_x)$ is a positive supermartingale
\begin{align*}
a>f(x) \geq E_{x} [f(X_{\sigma_B})] &\geq b\, P_{x}(\sigma_B<\infty)=b\,p_B(x).
\end{align*}
This contradicts $\mu(A)>0$. Therefore any excessive function is constant on $E$ $\mu$-a.e.\\
(c) Let $\psi(x)= E_x[1-e^{-\zeta}]$. Then
$$
p_t \psi(x)=E_x[\psi(X_t)\  ; \ t<\zeta]=E_x[1-e^{-(\zeta-t)}\  ; \  t<\zeta]\ \nearrow \ \psi(x)\quad as\  \ t\searrow 0.
$$
Hence $\psi$ is excessive and by (b) there is some $c\ge 0$ with $E_x[e^{-\zeta}]=c$ for $\mu$-a.e $x\in E$. 
By assumption, there are two finely open sets $G, \tilde G\subset E$, with $G\cap \tilde G=\emptyset$ and $\mu(G), \mu(\tilde G)>0$. Since $p_G=1$ on $G$, we obtain by (b)\\
\centerline{}
\centerline{ $p_G(x)=1$ for $\mu$-a.e. $x\in E$.}
\centerline{}
Since $\{\sigma_G<\infty\}=\{\sigma_G<\zeta\}$, we obtain for $\mu$-a.e. $x\in E$
\begin{align*}
c=E_x[e^{-\zeta}]&= E_x[e^{-\zeta}\  ;\  \sigma_G<\zeta]\\
&=E_x[e^{-\sigma_G} E_{X_{\sigma_G}}[e^{-\zeta}]\  ; \  \sigma_G<\zeta]\\
&=c\, E_x[e^{-\sigma_G}].
\end{align*}
But for $x\in \tilde{G}$, we have $E_x[e^{-\sigma_G}]<1$. Therefore $E_x[e^{-\zeta}]=0$ for $\mu$-a.e $x\in E$.\\
(d) Let $B$ be not $\mu$-polar and finely open. Using  (b), we get $p_B(x)=1$ for $\mu\text{-a.e.}\ x\in E$. Since $p_B(x)=P_x(L_B>0)$ for every $x\in E$, we get
$$
\psi(x):=P_x(0<L_B<\infty)=P_x(L_B<\infty) \text{ for } \mu\text{-a.e. } x\in E.
$$
Since $p_t \psi(x)=P_x(t<L_B<\infty)$ for any $x\in E$, we obtain that $\psi$ is excessive, bounded and $p_t \psi \to 0$ as $t\to \infty$. By (a) and (b), for some constant $c$
$$
\psi(x)=c\  \text{and} \ c=p_t c\text{ for } \mu\text{-a.e.}\ x\in E.
$$
But since $p_t \psi=p_t c \to 0$ as $t\to \infty$, we must have $c=0$, i.e. $P_x(L_B<\infty)=0$ for $\mu$-a.e. $x\in E$.
\end{proof}\\
\centerline{}
Suppose that the process $\mathbb{M}$ is associated with $\E$, i.e. $R_\alpha f$ is a $\mu$-version of $G_\alpha f$ for any $\alpha >0$, $f\in B(E)\cap L^2(E,\mu)$. Then the strict irreducibility and recurrence of $\tt$ implies the strict irreducibility and recurrence of $\mathbb{M}$. Consequently, by Proposition \ref{rpro}(d) for any non-empty and non-$\mu$-polar open set $B$,\\
\centerline{}
\centerline{$P_x(\Lambda)=1$ for $\mu$-a.e. $x\in E$,}
\centerline{}
where $\Lambda:=\{\omega \in \Omega:\ L_B(\omega)=\infty\}$. Furthermore, assume that the semigroup $p_t$ of $\mathbb{M}$ is strong Feller in the following sense: there exists a measurable function $\left (p_t(x,y)\right )_{t>0, x,y\in E}$ with
$$
E_x\Big [ f(X_t) \Big ] =p_t f(x)=\int_E p_t(x,y) f(y) \mu (dy)
$$
for any $x\in E$, $f\in B(E)_b$ and\\
\centerline{}
\centerline{$\displaystyle p_t f$ is continuous for any $f\in B(E)_b$.}
\centerline{}
Since $\Lambda$ is a shift invariant set, we can use the argument of \cite[Lemma 7.1]{DR} to see that\\
\centerline{}
\centerline{$\displaystyle P_x(\Lambda)=1$ for any $x$ in the support of $\mu$.}
\centerline{}
Consequently, for an arbitrary non-empty and non-$\mu$-polar open set $B$, the sample paths of $(X_t)_{t\geq 0}$ starting from any point $x$ in the support of $\mu$ come back to $B$ infinitely often. In particular, if $\mu$ has full support, then any non-empty open set $B$, satisfies $\mu(B)>0$ and is hence non-$\mu$-polar. Thus if $\mu$ has full support, for any non-empty open set $B$, the sample paths of $(X_t)_{t\geq 0}$ starting from any point $x$ in $E$ come back to $B$ infinitely often.

\section{Applications on Euclidean space}\label{3}
Throughout this section, we make the following assumptions: \\
Let $E\subset \R^d$ be either open or closed. If $E$ is closed, we assume $dx(\partial E)=0$ where $E$ is the disjoint union of its interior $E^0$ and its boundary $\partial E$. Let $\varphi \in L_{loc}^1 (E,dx)$ with $\varphi >0$ $dx$-a.e. and $d\mu :=\varphi dx$. Then $\mu$ is a $\sigma$-finite measure on $\B(E)$ and has full support. Let $C_0^\infty(E)$ be the set of infinitely often differentiable functions with compact support in $E$ if $E$ is open and $C_0^\infty(E):=\{ u\in E \longrightarrow \R :\ \exists \widetilde{u}\in C_0^\infty(\R^d)$ with $\widetilde{u}=u$ on $E \}$ if $E$ is closed. Let $\partial_i u$ denote the weak derivative of $u$ with respect to $x_i$, $\nabla u:=( \partial_1 u,\ldots,\partial_d u)$, $|\cdot|$ the Euclidean norm and $\langle\ ,\ \rangle$ the Euclidean inner product. \\
Consider $A=(a_{ij})_{1\leq i,j \leq d} \in L_{loc}^1(E,\mu)$ with symmetric part $\widetilde{a}_{ij}:=\frac{1}{2}(a_{ij}+a_{ji})$ and anti-symmetric part $\check{a}_{ij}:=\frac{1}{2}(a_{ij}-a_{ji})$ and suppose that for each relatively compact open set $V \subset E$, i.e. $V$ is relatively open in $E$ and its closure $\overline{V}$ is compact and contained in $E$, there exists $\nu_V >0$ such that
\begin{equation}\label{loc}
\nu_V^{-1} |\xi|^2 \leq \sum_{i,j=1}^d \widetilde{a}_{ij}(x) \xi_i \xi_j \leq \nu_V |\xi|^2
\end{equation}
for all $\xi \in \R^d$, $x\in V$. We assume further that\\
\centerline{}
\centerline{$\displaystyle \E^0(f,g):=  \int_E \langle A(x) \nabla f(x), \nabla g(x)\rangle \mu(dx),\ $ $f,g \in \ci$} \\
\centerline{}
is closable on $L^2(E,\mu)$ and that $(\E^0,C_0^\infty(E))$ satisfies the strong sector condition, i.e. there is a constant $K>0$ such that\\
\centerline{}
\centerline{$\displaystyle|\E^0(f,g)|\leq K \sqrt{\E^0(f,f)}\sqrt{\E^0(g,g)}$ for any $f,g \in \ci$.} \\
\centerline{}
Denote the closure of $(\E^0,\ci)$ on $L^2(E,\mu)$ by $(\E^0,D(\E^0))$. Then $(\E^0,D(\E^0))$ is a non-symmetric regular sectorial Dirichlet form on $L^2(E,\mu)$. \\
By $V\subset\subset E$, we mean that $V$ is relatively compact open in $E$. For $V\subset\subset E$, let $C_0^\infty(V):=\{u\in \ci :\ $supp$(u)\subset V\}$.  Since $(\E^0,C_0^\infty(E))$ is closable on $L^2(E,\mu)$, for any $V\subset\subset E$, $(\E^0,C_0^\infty(V))$ is closable on $L^2(V,\mu)$. Denote its closure by $(\E^{0,V},D(\E^{0,V}))$, then $D(\E^{0,V})\subset D(\E^0)$. Furthermore by (\ref{loc}), the $\E^0_1$-norm is equivalent to the norm $ \sqrt{\int_V (u^2+|\nabla u|^2) d\mu}$ on $D(\E^{0,V})$. Let
$$
D(\E^{0,E}):=\bigcup_{V \subset \subset E} D(\E^{0,V}),
$$
i.e. $f\in D(\E^{0,E})$, if and only if there exists a subset $V \subset \subset E$ such that $f \in D(\E^{0,V})$. Note that $D(\E^{0,E})\subset D(\E^{0}).$\\
Let $(L^0,D(L^0))$ be the linear operator corresponding to $(\E^0,D(\E^0))$ on $L^2(E,\mu)$. By \cite[I. Proposition 2.16]{MR}, we know that $D(L^0)=\{u\in D(\E^0) : v \longmapsto \E^0(u,v)$ is continuous with respect to $\sqrt{(v,v)}$ on $D(\E^0) \}$ and that $\E^0(f,g)=(-L^0 f,g)$ for any $f\in D(L^0)$, $g\in D(\E^0)$. Let $(T^0_t)_{t> 0}$ be the $C_0$-semigroup corresponding to $(L^0,D(L^0))$. \\
\centerline{}
Let $B:=(B_1,\ldots ,B_d)\in L_{loc}^2 (E,\R^d,\mu)$ be $\mu$-divergence free, i.e.
\begin{equation}\label{div1}
\int _E \langle B(x), \nabla f(x)\rangle \mu(dx)=0
\end{equation}
for any $f\in \ci$, hence for any $f\in D(\E^{0,E})$. Using the same technique as in \cite{St1}, we can construct a closed extension 
$(\overline{L},D(\overline{L}))$ of 
$$
Lu:=L^0 u+ \langle B,\nabla u \rangle, \ u\in D(L^0)_{0,b}
$$ 
on $L^1(E,\mu)$. For this, we need condition
\begin{itemize}
\item[(C)] $D(L^0)_{0,b}$ is a dense subset of $L^1(E,\mu)$,
\end{itemize}
which we assume from now on. 
\begin{rem}\label{opdenseness}
Condition (C) is needed to obtain strong continuity of the resolvent of $(\overline{L},D(\overline{L}))$, exactly as it is obtained in \cite{St1} right after display (1.15). 
It is a weak condition. For instance, consider $E:=\R^d$ and assume that
the coefficients of the generator of $L^{0}$ are locally square integrable with respect to the measure $\mu$ and that there are no boundary conditions. Then $C_0^{\infty}(E)\subset D(L^0)_{0,b}$, cf. e.g. Subsections \ref{exam1}, \ref{exam2} and Remark \ref{condC} below. Condition (C) can even be obtained when the coefficients are not locally integrable 
with respect to the measure $\mu$ (see end of Remark \ref{condC}). Similarly, one can obtain nice dense subsets of $D_0$ in case of boundary conditions. 
\end{rem}
\begin{lem}\label{prop}
There exists a closed operator $(\overline{L},D(\overline{L}))$ on $ L^1(E,\mu)$ which is the generator of a sub-Markovian $C_0$-semigroup of contractions $(\overline{T}_t)_{t> 0}$ satisfying the following properties:
\begin{itemize}
\item[(a)] $(\overline{L},D(\overline{L}))$ is a closed extension of $L u=L^{0}u+\langle B,\nabla u \rangle$, $u\in D(L^0)_{0,b}$ on $ L^1(E,\mu)$.
\item[(b)] $D(\overline{L})_b \subset D(\E^0)$ and for $u\in D(\overline{L})_b$, $v\in D(\E^{0,E})_{b}$, we have
$$
\E^0(u,v)-\int_E \langle B,\nabla u\rangle v d\mu=-\int_E \overline{L} u v d\mu
$$
and
$$
\E^0(u,u)\leq -\int_E \overline{L} u u d\mu.
$$
\end{itemize}
\end{lem}
Lemma \ref{prop} is proven in Section 4. Denote the  $C_0$-resolvent of $(\overline{L},D(\overline{L}))$ by $(\overline{G}_\alpha)_{\alpha>0}$. Since $(\overline{T}_t)_{t> 0}$ is a sub-Markovian $C_0$-semigroup of contractions on $L^1(E,\mu)$ and $L^1(E,\mu)_b \subset L^2(E,\mu)$ densely, we can construct uniquely a sub-Markovian $C_0$-semigroup of contractions $\tt$ on $L^2(E,\mu)$ such that $T_t \equiv \overline{T}_t$ for $t> 0$ on $L^1(E,\mu) \cap L^2(E,\mu)$ (cf. the Riesz-Thorin interpolation Theorem). Let $(L,D(L))$ be the generator of $\tt$ and $\ga$ be the corresponding $C_0$-resolvent.  Clearly, $G_\alpha \equiv \overline{G}_\alpha$ for $\alpha> 0$ on $L^1(E,\mu) \cap L^2(E,\mu)$. Let $(\widehat{L},D(\widehat{L}))$ be the adjoint operator of $(L,D(L))$ in $L^2(E,\mu)$.  Then
$$
\E(f,g):= \begin{cases}\ (-Lf,g)\qquad &\  f\in D(L),\ g\in L^2(E,\mu),\\
\ (-\widehat{L}g,f)\qquad &\  g\in D(\widehat{L}),\ f\in L^2(E,\mu),
\end{cases}
$$
is a generalized Dirichlet form on $L^2(E,\mu)$ according to Subsection \ref{2.1} with $\A\equiv 0$ on $\V=\H=L^2(E,\mu)$ and $(L,D(L))=(\Lambda,D(\Lambda))$ (see also \cite[I. Examples 4.9 (ii)]{St2}). Thus (H1) is satisfied. Clearly (H2) holds since $(\widehat{L},D(\widehat{L}))$ satisfies the same assumptions as $(L,D(L))$. In particular, the co-form\\
\centerline{}
\centerline{$\displaystyle\widehat{\E}(f,g):= \E(g,f)$ for $(f,g)\in D(\widehat{L})\times L^2(E,\mu) \cup L^2(E,\mu)\times D(L) $}
\centerline{}
is also a generalized Dirichlet form. Though in general $\E$ is neither symmetric nor sectorial, it has the same fundamental properties as $\widehat{\E}$. Moreover, the bilinear form $\E$ is an extension of
\begin{equation*}
\int_E \langle A \nabla f , \nabla g \rangle d\mu -  \int_E \langle B,\nabla f\rangle g d\mu, \  f,g\in \{ f\in D(L^0)_{0,b}\,:\,\langle B,\nabla f\rangle \in L^2(E,\mu)\}.
\end{equation*}
Since the $L^1(E,\mu)$-version $\ttb$ of $\tt$ is a sub-Markovian $C_0$-semigroup of contractions on $L^1(E,\mu)$, one can define recurrence and transience of $\tt$. Put
$$
Nv:=\langle B,\nabla v \rangle,\  v\in D(N):=D(\E^{0,E})_b.
$$
Then $D=D(N)\cap D(\E^0)=D(N)=D(\E^{0,E})_b$ and $\E$ satisfies assumption (H3). Indeed, if $u\in D(L)_b$ and $v\in D$, then there exists a function $f\in L^2(E,\mu)$ such that $u=G_1f$. We may assume that $f\geq 0$ $\mu$-a.e. Otherwise, we put $u=u^+-u^-$ where $u^+:=G_1 f^+$ and $u^-:=G^1 f^-$. Choose an increasing sequence of functions $(f_n)_{n\geq 1} \subset L^1(E,\mu)_b$ such that $0\leq f_n \nearrow f$ $\mu$-a.e. as $n \nearrow \infty$. Then $f_n \to f$ in $L^2(E,\mu)$ and $\overline{G}_1 f_n =G_1 f_n \to G_1 f$ in $L^2(E,\mu)$ as $n \to \infty$. Furthermore, since $\overline{G}_1 f_n$ is increasing in $n$, we obtain $u_n:=\overline{G}_1 f_n \leq G_1 f$ converges to $G_1 f$ $\mu$-a.e. as $n\to \infty$. Thus $(u_n)_{n\geq 1}\subset D(\overline{L})_b$ satisfies
$$
u_n \to u\ \text{in}\ L^2(E,\mu),\ \overline{L}u_n \to Lu\ \text{in}\ L^2(E,\mu),\ u_n \nearrow u\ \mu\text{-a.e. as}\ n \to \infty
$$
and $(u_n)_{n\geq 1}$ is uniformly bounded in $n$. Applying Lemma \ref{prop}, we can see that 
$$
\sup_{n \geq 1} \E^0(u_n,u_n)<\infty
$$
and so $u_n \to u$ weakly in $D(\E^0)$ as $n \to \infty$ as well as 
$$
\E^0(u,u)\leq \liminf_{n\rightarrow \infty}\E^0(u_n,u_n)
$$
by \cite[I. Lemma 2.12]{MR}. Hence using Lemma \ref{prop} and the approximation of $u$ with $(u_n)_{n\geq 1}$, we obtain
$$
\E^0(u,u)\leq \E(u,u),\ u\in D(L)_b
$$
and
\begin{equation}\label{luv}
(-Lu,v)= \E^0(u,v)+\int_E \langle B, \nabla v \rangle u d\mu,\ u\in D(L)_b,\  v\in D
\end{equation}
which achieves the proof that (H3) is satisfied. Consequently, by Theorem \ref{tit} and Corollary \ref{res} of Subsection \ref{2.1}, we get the following facts.
\centerline{}
\begin{cor}\label{cor} $ $
\begin{itemize}
\item[(a)] If $(\E^0,D(\E^0))$ is transient, then $\tt$ is also transient.
\item[(b)] If there exists a sequence of functions $(\chi_n)_{n\geq 1} \subset D$ with $0\leq \chi_n \leq 1$, $\lim_{n\rightarrow \infty} \chi_n =1$ $\mu$-a.e. satisfying
$$
\lim_{n \rightarrow \infty} \Big( \E^0(g,\chi_n) + \int_E \langle B, \nabla \chi_n \rangle g  d\mu\Big) =0,
$$
for any non-negative bounded $g$ in the extended Dirichlet space of $D(\E^0)$, then $\tt$ is not transient.
\end{itemize}
\end{cor}
\centerline{}
\begin{rem}\label{cor2}
If we can construct a sequence of functions $(\chi_n)_{n\geq 1} \subset D$ with $0\leq \chi_n \leq 1$, $\lim_{n\rightarrow \infty} \chi_n =1$ $\mu$-a.e. satisfying
$$
\lim_{n \rightarrow \infty} \Big( \E^0(\chi_n,\chi_n) + \int_E | \langle B, \nabla \chi_n \rangle|  d\mu\Big) =0,
$$
then $(\chi_n)_{n\geq 1}$ satisfies the conditions of Corollary \ref{cor}(b). Furthermore, since $-B$ satisfies the same assumptions as $B$, the co-form is then also not transient.
\end{rem}
\centerline{}
\centerline{}
Since $T_t \equiv \overline{T}_t$ for any $t> 0$ on $L^1(E,\mu) \cap L^2(E,\mu)$, it follows that the potential operator $G$ obtained from $\tt$ (see paragraph right before Definition \ref{rec1}) is equal to the potential operator obtained from $\ttb$ (cf. Definition \ref{rec1}(c)). Hence the recurrence $($resp. transience$)$ of $\tt$ is equivalent to the recurrence $($resp. transience$)$ of $\ttb$. Next, we want to show that the recurrence of $\ttb$ implies the existence of a nice sequence of functions $(\chi_n)_{n\geq 1}$. This will be achieved in Theorem \ref{rec3} below.\\
\centerline{}
\centerline{}
Let $h\in L^\infty(E,\mu)$, $h\geq 0$ $\mu$-a.e. and let $(\E^{0,h},D(\E^0))$ be the bilinear form on  $ L^2(E,\mu)$ defined by
\centerline{}
\centerline{$\displaystyle\E^{0,h}(f,g):= \E^0(f,g)+\int_E fgh d\mu,$ $f,g\in D(\E^0)$.}
\centerline{}
Since the $\E_1^{0,h}$- and $\E_1^0$-norms are equivalent on $D(\E^0)$, $(\E^{0,h},D(\E^0))$ is also a regular Dirichlet form on $ L^2(E,\mu)$. Let $(L^{0,h},D(L^{0,h}))$ be the generator of $(\E^{0,h},D(\E^0))$. Then $D(L^{0,h})=D(L^0)$ and $L^{0,h}u=L^0 u-h\cdot u$ for $u\in D(L^0)=D(L^{0,h}).$ The following construction Lemma \ref{con1} is also proven in Section 4.
\centerline{}
\centerline{}
\begin{lem}\label{con1}
There exists a closed operator $(\overline{L}^h,D(\overline{L}^h))$ on $ L^1(E,\mu)$ which is the generator of sub-Markovian $C_0$-resolvent of contractions $(\overline{G}_\alpha^h)_{\alpha>0}$ satisfying the following properties:
\begin{itemize}
\item[(a)] $(\overline{L}^h,D(\overline{L}^h))$ is a closed extension of $L^h u:=L^{0,h}u+\langle B,\nabla u \rangle$ $u\in D(L^0)_{0,b}$ on $ L^1(E,\mu)$.
\item[(b)] $D(\overline{L}^h)_b \subset D(\E^0)$ and for $u\in D(\overline{L}^h)_b$, $v\in D(\E^{0,E})_{b}$, we have\\
$$
\E^{0,h}(u,v)-\int_E \langle B, \nabla u \rangle v d\mu=-\int_E \overline{L}^h uv d\mu
$$
and
$$
\displaystyle  \E^{0,h}(u,u)\leq -\int_E \overline{L}^h uu d\mu.
$$
\item[(c)] $D(\overline{L}^h)=D(\overline{L})$ and for $f\in L^1(E,\mu)$ with $f\geq 0$\\
$$
\overline{G}_\alpha^h f=\overline{G}_\alpha(f-h \overline{G}_\alpha ^h f).
$$
\end{itemize}
\end{lem}
\centerline{}
Let $\varepsilon>0$ (be a constant) and let $h(\nequiv \e)$ be as in the paragraph preceding Lemma \ref{con1}. Consider the Hilbert space $L^2(E,(h+\varepsilon)\mu).$ Since
$$
\varepsilon \cdot (f,f) \leq (f,f)_{L^2((h+\varepsilon)\mu)} \leq ( \varepsilon+ \|h\|_{L^\infty(\mu)} ) \cdot (f,f)
$$
for any $f\in  L^2(E,\mu)$, $(\E^0,D(\E^0))$ is a regular Dirichlet form on $L^2(E,(h+\varepsilon)\mu)$ whose Dirichlet norm is equivalent to the norm of $(\E^0,D(\E^0))$ on $ L^2(E,\mu)$. Let $(L^{0,\varepsilon},D(L^{0,\varepsilon}))$ be the generator of $(\E^0,D(\E^0))$ on $L^2(E,(h+\varepsilon)\mu)$. Then $D(L^0)=D(L^{0,\e})$ and for $f\in D(L^0)$ and $g\in D(\E^0),$
$$
\E^0(f,g)=(-L^0f,g)=(-L^{0,\e}f,g)_{L^2((h+\e)\mu)}.
$$
It follows that $L^{0,\e}f= \frac{1}{h+\e}L^0 f$ for any $f\in D(L^{0,\e})$. For $V\subset \subset E$, since $L^2(\mu)$- and $L^2((h+\e)\mu)$-norms are equivalent and $(\E^0,C_0^\infty(E))$ is closable on $L^2(E,\mu)$, $(\E^0,C_0^\infty(V))$ is also closable on $L^2(V,(h+\e)\mu)$. Denote the closure of $(\E^0,C_0^\infty (V))$ on $L^2(V,(h+\e)d\mu)$ by $(\E^{0,V},D(\E^{0,\e,V}))$. Then it is easy to show that $D(\E^{0,\e,V})=D(\E^{0,V})$. Let
$$
B^\e(x):=\frac{1}{h(x)+\e}B(x).
$$
By (\ref{div1}),
\begin{equation*}
 \int_E \langle B^\e, \nabla f  \rangle(h+\e) d\mu=0
\end{equation*}
for any $f\in \ci$.
\centerline{}
\centerline{}
\begin{lem}\label{con2}
There exists a closed operator $(\overline{L}^\e,D(\overline{L}^\e))$ on $L^1(E,(h+\e)\mu)$ which is the generator of sub-Markovian $C_0$-resolvent of contractions $(\overline{G}_\alpha^\e)_{\alpha>0}$ satisfying the following properties:
\begin{itemize}
\item[(a)] $(\overline{L}^\e,D(\overline{L}^\e))$ is a closed extension of $L^\e u:=L^{0,\e}u+\langle B^\e,\nabla u \rangle$ $u\in D(L^{0,\e})_{0,b}$ on $L^1(E,(h+\e)\mu)$.
\item[(b)] $D(\overline{L}^\e)_b \subset D(\E^0)$ and for $u\in D(\overline{L}^\e)_b$, $v\in D(\E^{0,E})_{b}$, we have\\
$$
 \E^0(u,v)-\int_E \langle B^\e, \nabla u \rangle v (h+\e) d\mu=-\int_E \overline{L}^\e uv (h+\e) d\mu
$$
and
$$
\E^0(u,u)\leq -\int_E \overline{L}^\e uu (h+\e) d\mu.
$$
\item[(c)] $D(\overline{L}^\e)=D(\overline{L})$ and for $f\in L^1(E,(h+\e)\mu) $ with $f\geq 0$\\
$$
\overline{G}_\alpha^\e f=\overline{G}_\alpha((h+\e)f+\alpha(1-(h+\e))\overline{G}_\alpha^\e f).
$$
\end{itemize}
\end{lem}
\centerline{}
Lemma \ref{con2} is also proven in Section 4. Lemmas \ref{prop}, \ref{con1} and \ref{con2} assert that $D(\overline{L})=D(\overline{L}^h)=D(\overline{L}^\e)$ and for $u\in
 D(\overline{L})_b$, we have
$$
\int_E(\overline{L}u-\overline{L}^hu-hu)v d\mu =0, \  \int_E(\overline{L}u-(h+\e)\overline{L}^\e u)v d\mu=0,
$$
for any $v\in D(\E^{0,E})_b$. Since $D(\E^{0,E})_b \subset L^\infty(E, \mu)$ densely, we obtain for any $u\in D(\overline{L})_b,$\\
$$
\overline{L}^h u = \overline{L}u-h\cdot u\ \ \text{and}\ \ \overline{L}^\e u= \frac{1}{h+\e}\overline{L} u.
$$
\centerline{}
\begin{theo}\label{rec3}
If $\ttb$ is recurrent, then there exists a sequence of functions $(\chi_n)_{n\geq 1}$ in $D(\overline{L})_b$ with $0\leq \chi_n \leq 1$ and  $\lim_{n \rightarrow \infty} \chi_n =1$ $\mu$-a.e. satisfying $\lim_{n \rightarrow \infty} (-\overline{L} \chi_n, \chi_n)=0$. Furthermore, $\lim_{n \rightarrow \infty} -\overline{L} \chi_n=0$ $\mu$-a.e. and in $L^1(E,\mu)$. In particular, $(\widetilde{\E}^0,D(\E^0))$ is recurrent (see \cite[Theorem 1.6.3]{FOT} and Lemma \ref{prop}(b)).
\end{theo}
\begin{proof}
Let us choose $h\in L^1(E,\mu)_b$ with $h>0$ $\mu$-a.e. and $\e>0$. Then we know that by Lemma \ref{con1}, $\overline{G}^h_\e (\e f+fh)\in D(\overline{L})_b$ for any $f\in L^1(E,\mu)_b$ with $f\geq 0$. Observe
\begin{align*}
(1-\overline{L}^\e) \overline{G}_\e ^h (\e f+fh) &= \overline{G}^h_\e (\e f +f h)- \frac{1}{h+\e}\overline{L}\  \overline{G}^h_\e (\e f+ fh) \\
&= \overline{G}^h_\e (\e f +f h)- \frac{1}{h+\e}(\overline{L}^h-\e +h +\e )\overline{G}^h_\e (\e f+ fh) \\
&= \overline{G}^h_\e (\e f +f h)+ \frac{1}{h+\e}(\e f+fh)-\overline{G}^h_\e (\e f+ fh)=f 
\end{align*}
$\mu$-a.e. Consequently $\overline{G}^h_\e (\e f+fh)=\overline{G}^\e_1 f$. Thus if $0\leq f \leq 1$, then $0\leq \overline{G}^h_\e(\e f +fh)\leq 1$ for all $\e>0$.  Choosing $(f_n)_{n \geq 1} \subset L^1(E,\mu)_b,$ $f_n \geq 0$ for $n\geq 1$, $f_n \nearrow 1$ $\mu$-a.e. as $n\nearrow \infty$, we obtain
$$
0\leq \overline{G}_\e^h h=\lim_{n \rightarrow \infty} \overline{G}_\e^h(f_n h) \leq \limsup_{n \rightarrow \infty}\ \overline{G}_\e^h(\e f_n+f_nh)\leq 1.
$$
Letting $\e \to 0$, it follows $0\leq G^h h\leq 1$ $\mu$-a.e. where $G^h$ is the potential operator associated with $(\overline{G}^h_\alpha)_{\alpha>0}$.
Then using Lemma \ref{con1}(c), we get
$$
0\leq G(h(1-G^hh))=\lim_{\e \rightarrow 0} \overline{G}_\e (h(1-G^h h))= \lim_{\e \rightarrow 0}\overline{G}_\e (h-h \overline{G}^h_\e h)=\lim_{\e \rightarrow 0}\overline{G}_\e^h h =G^h h\leq 1.
$$
Since $\ttb$ is recurrent and $h(1-G^h h) \in L^1(E,\mu)_b$, hence by Definition \ref{rec1}(b) $G(h(1-G^h h))=0$ $\mu$-a.e. Consequently, $\overline{G}_1(h(1-G^h h))=0$ and by injectivity of $\overline{G}_1$, $G^h h=1$ $\mu$-a.e. If we put $\chi_n:= \overline{G}^h_{\frac{1}{n}}h$ for $n\geq 1$, then $0\leq \chi_n \leq 1$ and $\chi_n \nearrow 1$ $\mu$-a.e. as $n\nearrow \infty$. Moreover for all $n\geq 1$,
\begin{align*}
0\leq (-\overline{L} \chi_n, \chi_n) &= -\int_E \overline{L} \ \overline{G}^h_{\frac{1}{n}}h \chi_n d\mu = -\int_E \overline{L}\  \overline{G}_{\frac{1}{n}}(h-h\overline{G}^h_{\frac{1}{n}}h) \chi_n d\mu \\
&=-\frac{1}{n} \int_E \overline{G}_\frac{1}{n}(h-h\chi_n) \chi_n d\mu +\int_E (h-h\chi_n) \chi_n d\mu\\
&\leq \int_E h(1-\chi_n) d\mu
\end{align*}
and so $\lim_{n \rightarrow \infty} (-\overline{L} \chi_n, \chi_n)=0$.
\end{proof}\\
\centerline{}
\begin{defn}\label{cons1} $\ttb$ is said to be conservative if for some (and hence any) $t>0$, $\overline{T}_t 1=1$ $\mu$-a.e.
\end{defn}
\centerline{}
It is well known that if the Dirichlet form $(\E^0,D(\E^0))$ is strictly irreducible recurrent, then it is conservative (cf. \cite[Lemma 1.6.5]{FOT} and \cite[Corollary 1.3.8]{O13}). We have the following similar result in the non-sectorial situation of this section.\\
\begin{cor}\label{cons3} If $\ttb$ is recurrent, then it is conservative.
\end{cor}
\begin{proof}
Let $f\in L^1(E,\mu)_b$ with $f>0$. Then by Theorem \ref{rec3}, there exists $(\chi_n)_{n\geq 1}\subset D(\overline{L})_b$ such that $\lim_{n \rightarrow \infty} -\overline{L} \chi_n =0$ in $L^1(E,\mu)$. Consequently, we obtain
$$
\lim_{n \rightarrow \infty}\E(\chi_n, \widehat{G}_1 f)  =\lim_{n \rightarrow \infty}-\int_E \overline{L} \chi_n \widehat{G}_1 f d\mu=0.
$$
From this, the conservativeness of $\ttb$ follows by well-known standard arguments.
\end{proof}
\centerline{}
\centerline{}

\subsection{Explicit conditions for recurrence}\label{3.1}
Now, we shall find an explicit sequence of functions to determine recurrence of $\tt$. Assume that there exists a non-negative continuous function $\rho$ on $E$ with
$$
\nabla \rho \in L^\infty_{loc}(E,\R^d,\mu)
$$
such that for $r>0$
$$
E_r:=\{x\in E:\  \rho(x)< r\}
$$
is a relatively compact open set in $E$ and $\cup_{r> 0} E_r =E$. For instance, if $E$ is closed and so in particular if  $E=\R^d$, we may choose $\rho(x)=|x|$. Define for $ r > 0$,
\begin{equation}\label{v1}
v_1(r):= \int_{E_r} \langle A(x)\nabla \rho(x), \nabla \rho(x) \rangle \mu(dx).
\end{equation}
Since $E_r$ is increasing in $r$, we may assume that $v_1(r)> 0$ for $r>0$. From \cite[Theorem 3]{Stu1}, if
\begin{equation}\label{vr}
\int_1^\infty \frac{r}{v_1(r)}dr=\infty,
\end{equation}
then the symmetric Dirichlet form $(\widetilde{\E}^0,D(\E^0))$ is recurrent. Furthermore, starting from (\ref{vr}) we can explicitly construct a sequence of functions $(\chi_n)_{n\geq 1}\subset D(\E^{0,E})_{b}$ such that $0\leq \chi_n \leq 1,$ $\lim_{n \rightarrow \infty} \chi_n =1$ $\mu$-a.e. and $\lim_{n \rightarrow \infty} \E^0 (\chi_n, \chi_n)=0$. Indeed, let 
$$
a_n:= \int_1^n \frac{r}{v_1(r)}dr,
$$
then $a_n \geq 0$, $a_n$ is finite for all $n\geq 1$ and $\lim_{n\rightarrow \infty} a_n =\infty.$ Let
$$
\psi_n(r):= \begin{cases}\ 1 \qquad \qquad \qquad \qquad &0\leq r \leq 1,\\
\ \displaystyle 1-\frac{1}{a_n}\int_{1}^{r} \frac{t}{v_1(t)}dt &1\leq r \leq n,\\
\ 0&  n\leq r.
\end{cases}
$$
Then $\lim_{n \rightarrow \infty}\psi_n(r)=1$ $dr$-a.e. Let $\chi_n(x):= \psi_n(\rho(x))$. Since the support of $\psi_n(r)$ is $[0,n]$, the support of $\chi_n$ is $\overline{E}_n$. Similarly to \cite[Theorem 2.2]{GT}, we can show that $\chi_n \in D(\E^{0,E})_{b}$. We have $\nabla \chi_n(x)=-\frac{1}{a_n}1_{\overline{E}_n \setminus \overline{E}_1}(x)\frac{\rho(x)}{v_1(\rho(x))} \nabla \rho(x)$. Hence by the transformation theorem for $n\geq 1$,
$$
\E^0(\chi_n, \chi_n) = \int_{E_n\setminus E_1} \langle A(x) \nabla \chi_n(x), \nabla \chi_n(x)\rangle  \mu(dx)=\frac{1}{a_n^2} \int_1^n  \frac{r^2}{v_1(r)^2} \nu_1(dr)
$$
where $\nu_1$ is the unique measure on $([0,\infty), \B([0,\infty)))$ which has $v_1$ as distribution function. Let $\eta$ be a standard mollifier on $\R$. Set $\eta_\e(r)=\frac{1}{\e}\eta(\frac{r}{\e})$ so that $\int_\R \eta_\e(r) dr=1$. Let
$$
v_1^\e (r):= \int_\R v_1(r-t)\eta_\e(t)dt.
$$
Then since $v_1$ is continuous and strictly increasing, $v_1^\e$ is also continuous and strictly increasing and $v_1^\e$ uniformly converges to $v_1$ as $\e \to 0$ on each compact set in $[0,\infty)$. Let $\nu_1^\varepsilon$ be the unique measure on $([0,\infty), \B([0,\infty)))$ which has $v_1^\e$ as distribution function. Then, for any continuous function $f$, we have
$$
\int_1^n f(r) \nu_1^\e (dr)=\int_1^n f(r) v_1^\e(r)' dr \longrightarrow \int_1^n f(r) \nu_1(dr)
$$
as $\varepsilon \to 0$. Consequently,
\begin{align*}
\int_1^n \frac{r^2}{v_1(r) ^2} \nu_1(dr)&= \lim_{\e\rightarrow 0} \int_1^n \frac{r^2}{v_1(r) ^2}v_1 ^\e(r)'dr \\
&=\lim_{\e\rightarrow 0} \int_1^n \frac{r^2}{v_1 ^\e(r) ^2}v_1 ^\e(r)'dr \\
&=\lim_{\e\rightarrow 0} \int_1^n r^2 \frac{d}{dr}\Big ( \frac{-1}{v_1^\e(r)} \Big) dr\\
&=2\int_1^n \frac{r}{v_1(r)}dr +\frac{1}{v_1 (1)}-\frac{n^2}{v_1(n)}.
\end{align*}
Thus, $\E^0(\chi_n,\chi_n)\leq \frac{2}{a_n}+\frac{1}{a_n ^2 v_1(1)}$. Since the last term tends to $0$ as $n\to \infty$, there exists a sequence of functions $(\chi_n)_{n\geq 1} \subset D(\E^{0,E})_{b} $ with $0\leq \chi_n \leq 1$, $\lim_{n\rightarrow \infty} \chi_n =1$ $\mu$-a.e. satisfying
$$
\lim_{n \rightarrow \infty} \E^0(\chi_n,\chi_n) =0.
$$
\centerline{}
Now, we present an explicit sufficient condition for recurrence of $\tt$.  Let
\begin{equation}\label{v2}
v_2(r):= \int_{E_r} \rho(x)\cdot | \langle B(x) ,\nabla \rho(x)\rangle | \mu(dx)
\end{equation}
and $\nu_2$ be the measure on $([0,\infty), \B([0,\infty)))$ which has $v_2$ as distribution function. Let
\begin{equation}\label{v1p2}
v(r):=v_1(r)+v_2(r)
\end{equation}
and $\nu$ be the measure on $([0,\infty), \B([0,\infty)))$ which has $v$ as distribution function. Then it is easy to see that $\nu(A)\geq \nu_i(A)$ for $A\in \B([0,\infty))$, $i=1,2$.\\
\begin{theo}\label{ee}
Let $v_1$, $v_2$, $v$ be defined as in (\ref{v1}), (\ref{v2}) and (\ref{v1p2}). If the sequence $(a_n)_{n\ge 1}$ defined by
$$
a_n:=\int_1^n \frac{r}{v(r)}dr, \ n\ge 1,
$$
satisfies $\lim_{n \rightarrow \infty}a_n=\infty$ and $\lim_{n \rightarrow \infty} \frac{\log (v_2(n)\vee 1)}{a_n}=0$, 
then $\tt$ is not transient. In particular, if $\tt$ is additionally strictly irreducible, then  $\tt$ is recurrent.
\end{theo}
\begin{proof} In view of Corollary \ref{res}(b), the last assertion follows from the first one. Concerning the first one, it follows from Remark \ref{cor2}, that it suffices to construct a sequence of functions $(\chi_n)_{n\geq 1} \subset D(\E^{0,E})_{b}$ with $0\leq \chi_n \leq 1$, $\lim_{n\rightarrow \infty} \chi_n =1$ $\mu$-a.e. satisfying
\begin{equation}\label{ee1}
\lim_{n \rightarrow \infty} \Big( \E^0(\chi_n,\chi_n) + \int_{E} | \langle B, \nabla \chi_n \rangle|  d\mu\Big) =0.
\end{equation}
First assume that $B$ is not identically zero with respect to $\mu$. For $r>0$, let
$$
\psi_n(r):= \begin{cases}\ 1 \qquad \qquad \qquad \qquad &0\leq r \leq 1,\\
\ \displaystyle 1-\frac{1}{a_n}\int_{1}^{r} \frac{t}{v(t)}dt &1\leq r \leq n,\\
\ 0&n\leq r.
\end{cases}
$$
Then $\lim_{n \rightarrow \infty}\psi_n(r)=1$ $dr$-a.e. Let $\chi_n(x):= \psi_n(\rho(x))$. Then $\chi_n \in D(\E^{0,E})_{b}$. We have $\nabla \chi_n(x)=-\frac{1}{a_n}1_{\overline{E}_n \setminus \overline{E}_1}(x)\frac{\rho(x)}{v(\rho(x))}\nabla \rho(x)$. Hence for $n\geq 1$,
\begin{align*}
\E^0(\chi_n, \chi_n)+\int_{E} |\langle B, \nabla \chi_n \rangle | d\mu &= \int_{E_n\setminus E_1} \langle A(x) \nabla \chi_n(x), \nabla \chi_n(x)\rangle +|\langle B(x),\nabla \chi_n(x)\rangle | \mu(dx)\\
&=\frac{1}{a_n^2} \int_{E_n\setminus E_1}    \frac{\rho(x)^2}{v(\rho(x))^2}\langle A(x)\nabla \rho(x),\nabla \rho(x)\rangle\mu(dx)\\
&\qquad  + \frac{1}{{a_n}} \int_{E_n\setminus E_1} \frac{\rho(x)}{v(\rho(x))}|\langle B(x),\nabla \rho(x) \rangle| \mu(dx) \\
&=\frac{1}{a_n^2} \int_1^n \frac{r^2}{v(r)^2}\nu_1(dr) +\frac{1}{a_n}\int_1^n \frac{1}{v(r)}\nu_2(dr)\\
&\leq \frac{1}{a_n^2} \int_1^n \frac{r^2}{v(r)^2}\nu(dr) +\frac{1}{a_n}\int_1^n \frac{1}{v_2(r)}\nu_2(dr)\\
&\leq \frac{2}{a_n}+\frac{1}{a_n^2 v(1)}+ \frac{\log (v_2(n)\vee 1)}{a_n}.
\end{align*}
By our assumptions, the last term tends to $0$ as $n\to \infty$. Consequently $\tt$ is recurrent. If $B\equiv 0$ $\mu$-a.e., then $\log( v_2(n)\vee 1)\equiv 0$ and (\ref{ee1}) also holds.
\end{proof}
\begin{cor}\label{result} Let $v_1$, $v_2$, and $v$ be defined as in (\ref{v1}), (\ref{v2}), and (\ref{v1p2}). 
The conditions on $(a_n)_{n\ge 1}$ in Theorem \ref{ee} are satisfied, if one of the following conditions is fulfilled for sufficiently large $r$:
\begin{itemize}
\item[(a)] $v_1(r) \leq b r^2$ and $v_2(r) \leq b\log r$ for some constant $b>0$,
\item[(b)] $v(r) \leq c r^\alpha$ for some constants $c>0$ and $ \alpha <2$.
\end{itemize}
Consequently, if either (a) or (b) holds, then $\tt$ is not transient. In particular, if $\tt$ is additionally strictly irreducible, then  $\tt$ is recurrent.
\end{cor}
\subsection{Examples and counterexamples}\label{3.2}
In this Subsection, we provide explicit examples and counterexamples. We start with several counterexamples which show that the existence of ($\chi_n)_{n \geq 1}\subset D(\E^0)$ such that  $0\leq \chi_n \leq 1$, $\lim_{n \rightarrow \infty}\chi_n =1$ $\mu$-a.e. and $\lim_{n \rightarrow \infty} \E(\chi_n,\chi_n)=0$ is not a sufficient condition for recurrence of $\tt$ in contrast to the symmetric case where this is always true (cf. \cite[Theorem 1.6.3]{FOT}). At the end of this Subsection, we discuss recurrence and transience related to Muckenhoupt weights.
\subsubsection{A counterexample using results from \cite{St1}}\label{exam1}
Consider the case where $E=\R$ and $(\E^0,D(\E^0))$ is given as the closure of
$$
\E^0(f,g):=\int_\R f'(x)g'(x)\mu(dx),\  f,g\in C_0^\infty(\R)
$$
on $L^2(\R,\mu)$ where $d\mu:=e^{-x^2}dx$. Then it is easy to see that $1\in D(L^0)$, $L^0 1=0$ and $C_0^\infty(\R)\subset D(L^0)$. In particular, condition (C) is satisfied. Moreover, $B(x):=-6e^{x^2}$ satisfies $(\ref{div1})$ and so by Lemma \ref{prop}, we can construct a closed operator $(\overline{L},D(\overline{L}))$ which is a closed extension of $Lu:=L^0 u+Bu'$, $u\in D(L^0)_{0,b}$ on $L^1(\R,\mu)$ satisfying (a) and (b) in Lemma \ref{prop}. By \cite[Remark 1.11 and Example 1.12]{St1}, $\ttb$ is not conservative, hence not recurrent by Corollary \ref{cons3}.\\
Since $\mu (\R)<\infty$, the restriction of $\ttb$ on $L^2(\R,\mu)$ coincides with the $L^2(\R,\mu)$-semigroup $\tt$. Thus $(L,D(L))$ is given as the part of ($\overline{L},D(\overline{L}))$ on $L^2(\R,\mu)$, i.e.
$$
D(L)=\{u\in L^2(\R,\mu)\cap D(\overline{L}):\  \overline{L}u\in L^2(\R,\mu)\}
$$
and
$$
Lf:=\overline{L} f \quad f\in D(L).
$$
Let $D:=D(\E^0)_{0,b}.$ Then for $f\in D(L)_b$, $g\in D$, we have by (\ref{luv}),\\
$$
\E(f,g)=(-Lf,g)=\E^0(f,g)+\int_\R Bg'f d\mu.
$$
and
$$
\E^0(f,f)\leq \E(f,f)
$$
where $g'$ denotes the derivative of $g$. Thus $\E$ satisfies (H1)-(H3). By construction of $(\overline{L},D(\overline{L}))$, we have
$$
D(L^0)_{0,b}\subset D(\overline{L})_{0,b}
$$
and if $u \in D(L^0)_{0,b}$, then $u\in L^2(\R,\mu)$ and
$$
\overline{L}u=L^0u+Bu'\in L^2(\R,\mu).
$$
Consequently, $C_0^\infty (\R)\subset D(L^0)_{0,b}\subset D(L)_{0,b}$. Choose $(\chi_n)_{n \geq 1} \subset C_0^\infty(\R)$ such that $0\leq \chi_n\leq 1$, $\lim_{n\rightarrow \infty} \chi_n=1$ $\mu$-a.e. and $\| \chi'_n\|_{L^\infty(\mu)} \leq 2/n$. It then follows from (\ref{div1}), that
$$
\lim_{n \rightarrow \infty}(-L\chi_n,\chi_n)=\lim_{n \rightarrow \infty}\E(\chi_n,\chi_n)=\lim_{n \rightarrow \infty}\E^0(\chi_n,\chi_n)=0.
$$
\centerline{}
\subsubsection{A generic counterexample}\label{exam2}
We call the following counterexample generic, since it works for a large class of $\varphi$. We let hence 
$E=\R$, $\varphi:\R \to \R^+$ be locally bounded above and below by strictly positive constants with $\varphi ' \in L^2_{loc}(\R,dx)$, $d\mu=\varphi dx$ and $B(x)=\frac{b}{\varphi(x)}$ for some constant $b\neq 0$. Note that these general assumptions on $\varphi$ imply that $C_0^{\infty}(\R)\subset D(L^0)_{0,b}\subset D(L)_{0,b}$ and that 
$$
Lf=\frac12 f''+\left (\frac{\varphi'}{2\varphi}+B\right)f',\ f\in C_0^{\infty}(\R).
$$
These two facts are important for our arguments below. In particular, condition (C) is satisfied.\\
Using similar arguments as in Subsection \ref{exam1},  we can construct a generalized Dirichlet form $\E$ satisfying (H1)-(H3) and such that $\E$ is given as an extension of
$$
\E(f,g):=\frac{1}{2}\int_\R f'(x)g'(x)\mu(dx)-\int_\R B(x)f'(x)g(x)\mu(dx),\  f,g\in C_0^\infty(\R)
$$
on $L^2(\R,\mu)$. By the specialties of dimension one, $\E$ can be symmetrized, i.e. there exists a symmetric Dirichlet form $(\widetilde{\E},D(\widetilde{\E}))$ whose infinitesimal generator $(\widetilde{L},D(\widetilde{L}))$ coincides with ($L,D(L))$ locally. This will be realized in (\ref{gloc}) below.\\
For $n\geq 1$, let $V_n:=(-n,n)$ be the open interval from $-n$ to $n$ in $\R$ and $(\E^{0,V_n},D(\E^{0,V_n}))$ be the symmetric Dirichlet form given as the closure of
$$
\E^{0,V_n}(f,g)=\frac{1}{2}\int_{V_n} f'g'd\mu,\ f,g\in C_0^\infty(V_n).
$$
Let ($L^{0,V_n},D(L^{0,V_n}))$ be the closed linear operator on $L^2(V_n,\mu)$ corresponding to $(\E^{0,V_n},D(\E^{0,V_n}))$. Since $B$ satisfies (\ref{div1}), by \cite[Proposition 1.1]{St1}, we can construct a closed operator $(\overline{L}^{V_n},D(\overline{L}^{V_n}))$ which is the closure of $L^{V_n}u=L^{0,V_n}u+Bu'$, $u \in D(L^{0,V_n})_b$ on $L^1(V_n,\mu)$. Let $(L^{V_n},D(L^{V_n}))$ be the part of $(\overline{L}^{V_n},D(\overline{L}^{V_n}))$ on $L^2(V_n,\mu)$, then we have $D(L^{0,V_n})_b\subset D(L^{V_n})$,
\begin{equation}\label{lvn}
L^{V_n}f=\overline{L}^{V_n}f=L^{0,V_n}f+Bf',\ f\in D(L^{0,V_n})_b
\end{equation}
and
$$
\E^{0,V_n}(f,g)-\int_{V_n}Bf'gd\mu=-\int_{V_n} L^{V_n}fgd\mu,\  f\in D(L^{V_n})_b,\ g\in D(\E^{0,V_n}).
$$
Let ($G_\alpha^{V_n})_{\alpha>0}$ be the $C_0$-resolvent of contractions corresponding to $(L^{V_n},D(L^{V_n}))$. Since the $L^2(\mu)$- and $L^2(dx)$-norms are equivalent on $V_n$, $D(\E^{0,V_n})=H_0^1(V_n):=$ the closure of $C_0^\infty(V_n)$ with respect to the norm $\sqrt{\int_{V_n}(u^2+(u')^2)dx}$ in $L^2(V_n,dx)$. Thus, $u\in D(\E^{0,V_n})$, if and only if $u$ is equal a.e. to an absolutely continuous function which has a.e. an ordinary derivative belonging to $L^2(V_n,dx)$ and $u$ does not have a boundary value, i.e. for any $v\in C_0^\infty(\overline{V}_n)$,
$$
\int_{V_n} u'vdx =-\int_{V_n} uv'dx.
$$
Let
\begin{equation}\label{phitilde}
\widetilde{\varphi}(x):= \exp \Big (\int_0^x \frac{\varphi'(s)+2b}{\varphi(s)}ds \Big )
\end{equation}
and  $(\widetilde{\E}^{V_n},D({\E}^{0,V_n}))$ be the bilinear form on $L^2(V_n,\widetilde{\varphi}dx)$ defined by
$$
\widetilde{\E}^{V_n}(f,g):=\frac{1}{2}\int_{V_n} f'g'\widetilde{\varphi}dx,\ f,g\in D({\E}^{0,V_n}).
$$
Since the $L^2(\mu)$- and $L^2(\widetilde\varphi dx)$-norms are equivalent on $V_n$, $(\widetilde{\E}^{V_n},D({\E}^{0,V_n}))$ is a symmetric Dirichlet form on $L^2(V_n,\widetilde{\varphi}dx)$.
Let ($\widetilde{L}^{V_n},D(\widetilde{L}^{V_n}))$ be the closed linear operator on $L^2(V_n,\widetilde{\varphi} dx)$ corresponding to $(\widetilde{\E}^{V_n},D({\E}^{0,V_n}))$.
\begin{lem}\label{cexam}
$D(\widetilde{L}^{V_n})=D(L^{0,V_n})$ and for $u\in D(\widetilde{L}^{V_n})$,
$$
\widetilde{L}^{V_n}u=L^{0,V_n}u+Bu'.
$$
\end{lem}
\begin{proof}
Suppose that $u\in D(L^{0,V_n})$. We first show that $u \in D(\widetilde{L}^{V_n})$, i.e. $v \longmapsto \widetilde\E^{V_n}(u,v)$ is continuous with respect to 
$\sqrt{(v,v)_{L^2(V_n, \widetilde\varphi dx)}}$ 
on $D(\E^{0,V_n})$. Let
$$
\psi(x):=\exp\Big ( \int_0^x \frac{2b}{\varphi(s)}ds\Big ).
$$
It is easy to see that $\widetilde{\varphi}=\psi\varphi$ and that $v\psi\in D(\E^{0,V_n})$ for any $v\in D(\E^{0,V_n})$. Then
$$
\E^{0,V_n}(u,v\psi)=\frac{1}{2}\int_{V_n} u'(v\psi)'\varphi dx =\frac{1}{2}\int_{V_n} u'v'\psi\varphi dx+\frac{1}{2}\int_{V_n} u'\psi' v \varphi dx.
$$
Consequently, we obtain
$$
\widetilde{\E}^{V_n}(u,v)=-\int_{V_n} L^{0,V_n}u\cdot v\psi \varphi dx-\int_{V_n} \frac{b}{\varphi}u'v\psi \varphi dx,
$$
hence $u\in D(\widetilde{L}^{V_n})$ and
$$
\widetilde{L}^{V_n}u=L^{0,V_n}u+Bu'.
$$
Conversely, suppose that $u \in D(\widetilde{L}^{V_n})$ and that $v\in D(\E^{0,V_n})$. Then $\frac{v}{\psi} \in D(\E^{0,V_n})$ and
$$
\widetilde{\E}^{V_n}(u,\frac{v}{\psi})=\frac{1}{2}\int_{V_n}u'\Big (\frac{v}{\psi}\Big )'\widetilde\varphi dx=\frac{1}{2}\int_{V_n}u'v'\varphi dx-\frac{1}{2}\int_{V_n}u'v\frac{\psi '}{\psi^2}\widetilde\varphi dx.
$$
Consequently, we obtain
$$
\E^{0,V_n}(u,v)=-\int_{V_n}\widetilde{L}^{V_n}u\cdot \frac{v}{\psi}\widetilde{\varphi}dx+\int_{V_n}\frac{b}{\varphi}u'\frac{v}{\psi}\widetilde{\varphi}dx
$$
and so $u\in D(L^{0,V_n})$.
\end{proof}\\
\centerline{}
Let ($\widetilde{G}^{V_n}_\alpha)_{\alpha>0}$ be the $C_0$-resolvent of contractions corresponding to ($\widetilde{L}^{V_n},D(\widetilde{L}^{V_n}))$. By Lemma \ref{cexam} and (\ref{lvn}), we obtain $D(\widetilde{L}^{V_n})=D(L^{0,V_n})$ and for any $u \in D(L^{0,V_n})_b$,
$$
L^{V_n}u=\widetilde{L}^{V_n} u.
$$
Since ($\widetilde{L}^{V_n},D(\widetilde{L}^{V_n}))$ is a Dirichlet operator on $L^2(V_n,\widetilde\varphi dx)$, we get $D(\widetilde{L}^{V_n})_b\subset D(\widetilde{L}^{V_n})$ densely and so by Lemma \ref{cexam},
$$
(1-L^{V_n})D(L^{0,V_n})_b =(1-\widetilde{L}^{V_n})D(\widetilde{L}^{V_n})_b\subset L^2(V_n,\widetilde\varphi dx)=L^2(V_n,\mu)
$$
densely. Consequently, we obtain $D(\widetilde{L}^{V_n})=D(L^{V_n})$ and for $u\in D(L^{V_n})$,
$$
L^{V_n}u=\widetilde{L}^{V_n}u.
$$
It follows that for $f \in L^2(V_n,\mu)$,
\begin{equation}\label{tga}
G_\alpha^{V_n}f=(\alpha-L^{V_n})^{-1}f=(\alpha-\widetilde{L}^{V_n})^{-1}f=\widetilde{G}_\alpha^{V_n}f,\ \mu\text{-a.e.}
\end{equation}
The $C_0$-resolvent of contractions $(\overline{G}_\alpha)_{\alpha>0}$ of $(\overline{L},D(\overline{L}))$ is defined by
$$
\overline{G}_\alpha f=\lim_{n \to \infty}\overline{G}_\alpha ^{V_n}(f\cdot 1_{V_n}),\ \mu\text{-a.e.}\ f\in L^1(\R,\mu)
$$
(cf. proof of Lemma \ref{prop}). For the $C_0$-resolvent of contractions $\ga$ of $(L,D(L))$, it holds (see right after Lemma \ref{prop}) 
$$
G_\alpha f=\overline{G}_\alpha f= \lim_{n \to \infty}\overline{G}^{V_n}_\alpha (f\cdot 1_{V_n})=\lim_{n\to \infty} G_\alpha^{V_n}(f\cdot 1_{V_n}),\ \mu\text{-a.e.}\ f\in L^1(\R,\mu)\cap L^2(\R,\mu),
$$
hence $G_\alpha f=\lim_{n \to \infty}G_\alpha^{V_n}(f \cdot 1_{V_n})$, $f\in L^2(\R,\mu)$.\\
Next, we will construct a symmetric Dirichlet form on $L^2(\R,\widetilde\varphi dx)$ which extends $(\widetilde\E^{V_n},D(\E^{0,V_n}))$ for any $n\geq 1$. We have already constructed a sub-Markovian $C_0$-resolvent of contractions $(\widetilde{G}_\alpha^{V_n})_{\alpha>0}$ on $L^2(V_n,\widetilde\varphi dx)$. For $f\in L^2(\R,\widetilde\varphi dx)$, $\alpha>0$
$$
\widetilde{G}_\alpha f:=\lim_{n \to \infty} \widetilde{G}_\alpha^{V_n}(f\cdot 1_{V_n})
$$
exists $\widetilde\varphi dx$-a.e. and $(\widetilde{G}_\alpha)_{\alpha>0}$ is a sub-Markovian $C_0$-resolvent of contractions on $L^2(\R,\widetilde\varphi dx)$ (cf. proof of Lemma \ref{prop}). Since for each $n\geq 1$, ($\widetilde{G}_\alpha^{V_n})_{\alpha >0}$ is symmetric, so is $(\widetilde{G}_\alpha)_{\alpha>0}$. Let $(\widetilde\E,D(\widetilde\E))$ be the symmetric Dirichlet form corresponding to $(\widetilde{G}_\alpha)_{\alpha>0}$. Then $(\widetilde\E,D(\widetilde\E))$ is a closed extension of 
$$
\frac{1}{2}\int_\R f'g'\widetilde \varphi dx,\ f,g \in C_0^\infty(\R).
$$
For $f\in L^2(\R,\mu)\cap L^2(\R,\widetilde{\varphi}dx)$, using the above and (\ref{tga}) it holds
\begin{equation}\label{gloc}
G_\alpha f=\lim_{n \to \infty} G_\alpha^{V_n} (f\cdot 1_{V_n})=\lim_{n \to \infty}\widetilde{G}^{V_n}_\alpha (f\cdot 1_{V_n})=\widetilde{G}_\alpha f, \ \mu\text{-a.e.}
\end{equation}
Therefore, the potential operators of $(G_\alpha)_{\alpha>0}$ and ($\widetilde{G}_\alpha)_{\alpha>0}$ are the same on $L^1(\R,\mu)\cap L^1(\R,\widetilde{\varphi}dx)$ and the recurrence or transience of $(G_\alpha)_{\alpha>0}$ and ($\widetilde{G}_\alpha)_{\alpha>0}$ are equivalent.
\begin{rem}\label{finalrem}
If we choose $\varphi$ and $b$ so that either it holds
\begin{equation}\label{criterion}
\int_0^\infty \frac{1}{\widetilde{\varphi}(x)}dx<\infty \ \ \text{or} \ \ \int_{-\infty}^0 \frac{1}{\widetilde{\varphi}(x)}dx<\infty,
\end{equation}
where $\widetilde{\varphi}$ is as in (\ref{phitilde}), then it follows similarly to \cite[Theorem 3.11]{ORT} that $\widetilde{\E}$ is not recurrent. Consequently, $\E$ is also not recurrent. However, as in Subsection \ref{exam1}, there exists a sequence of functions $(\chi_n)_{n \geq 1} \subset C_0^\infty(\R)$, such that $0\leq \chi_n\leq 1$, $\lim_{n\rightarrow \infty} \chi_n=1$ $\mu$-a.e. and
$$
\lim_{n \rightarrow \infty}\E(\chi_n,\chi_n)=0.
$$
\end{rem}
For instance, if $\varphi(x)=e^{-|x|}$, $b=\frac12$, then $\widetilde{\varphi}(x)=e^{-x} \exp ((e^x-1))$ for $x\geq 0$. Consequently
$$
\int_0^\infty \frac{1}{\widetilde{\varphi}(x)}dx=\int_0^\infty \frac{e^x}{\exp(e^x-1)}dx=1,
$$
and so the criterion (\ref{criterion}) of Remark \ref{finalrem} is satisfied. Moreover, it is easy to see that for this choice of $\varphi$ and $b$, $\E$ has the following additional properties: $\E$ is not conservative and $\E$ does not satisfy the weak sector condition, i.e. it holds
\[
\sup_{u,v\in C_0^\infty(\R)\setminus \{0\}}\frac{|(-Lu,v)|}{\|u\|_{D(\E^{0})}\|v\|_{D(\E^{0})} }=\infty.
\]
Replacing $\varphi(x)=e^{-|x|}$ by $\varphi(x)=\min\{1,\frac{1}{|x|}\}$ the criterion (\ref{criterion}) of Remark \ref{finalrem} is still satisfied, but $\E$ becomes conservative and does not satisfy the strong 
sector condition, i.e. it holds
\[
\sup_{u,v\in C_0^\infty(\R)\setminus \{0\}}\frac{|(-Lu,v)|}{\sqrt{\E^{0}(u,u)}\sqrt{\E^{0}(v,v)}}=\infty.
\]
However, in this case, it is not easy to see whether $\E$ satisfies the weak sector condition or not.

\begin{exam}\label{exam22}
Choosing $\varphi(x)\equiv 1$ and $B(x)\equiv b$ for some constant $b\neq 0$, gives another example where 
the criterion (\ref{criterion}) of Remark \ref{finalrem} is satisfied. Hence $\E$ is not recurrent, but there exists a sequence of functions $(\chi_n)_{n \geq 1} \subset C_0^\infty(\R)$ such that $0\leq \chi_n\leq 1$, $\lim_{n\rightarrow \infty} \chi_n=1$ $\mu$-a.e. and
\begin{equation}\label{en0}
\lim_{n \rightarrow \infty}\E(\chi_n,\chi_n)=0.
\end{equation}
However, by \cite[Proposition 1.10(c)]{St1}, $dx$ is $(\overline{T}_t)$-invariant. This example shows that even though (\ref{en0}) holds and the reference measure $dx$ is $(\overline{T}_t)$-invariant, $\tt$ does not need to be recurrent. Obviously, in this example $\E$ satisfies the weak sector condition, but not the strong sector condition, i.e. $\E$ is not sectorial in the sense of this paper.
\end{exam}
\subsubsection{Muckenhoupt weights}\label{exam3}
In this Subsection, we present a class of examples of $\varphi$ and $B$ applying Corollary \ref{result} and Corollary \ref{cor}(a). We consider the case where $E=\R^d$ with $d\geq 2$ and $(\E^0,D(\E^0))$ is given as the closure of
$$
\E^0(f,g):=\int_{\R^d} \langle \nabla f, \nabla g\rangle d\mu,\  f,g\in C_0^\infty(\R^d)
$$
on $L^2(\R^d,\mu)$, where $d\mu:=\varphi dx$ and $\varphi$ is an $\cal{A}$$_\beta$-weight, $\beta\in[1,2]$ (cf. \cite[Definition 1.2.2]{Tu}). Note that for $\varphi\in\cal{A}_\beta$ (short for $\varphi$ is an $\cal{A}_\beta$-weight), $\beta\in[1,2]$, the closability follows since $\cal{A}_\beta\subset {\cal A}$$_{2}$ and $\frac{1}{\varphi}\in L^1_{loc}(\R^d,dx)$ for any $\varphi\in\cal{A}$$_2$ (cf. \cite[Remark 1.2.4]{Tu}). Assume that $B\in L^2_{loc}(\R^d,\R^d,\mu)$ satisfies $(\ref{div1})$ and that $D(L^0)_{0,b}$ is a dense subset of $L^1(\R^d,\mu)$, i.e. condition (C) is satisfied. 
\begin{rem}\label{condC} For instance, if $\varphi=\xi^2$, $\xi\in H_{loc}^{1,2}(\R^d,dx)$, $\varphi>0$ $dx$-a.e. where $H^{1,2}(\R^d,dx)$ denotes the usual Sobolev space of order one in $L^2(\R^d,dx)$ and $H_{loc}^{1,2}(\R^d,dx):=\{f: f\cdot \chi\in H^{1,2}(\R^d,dx)\text{ for any } \chi \in C_0^\infty(\R^d) \}$, then $C_0^\infty(\R^d)\subset D(L^0)$ and (C) holds. Another example is given in \ref{exam3}(c) below, where the drift coefficient may even not be in $L^1_{loc}(\R^d,\mu)$, i.e. in the non-semimartingale case.
\end{rem}
Under the present assumptions, we can construct as before a generalized Dirichlet form $\E$ satisfying (H1)-(H3) and which is an extension of
$$
\int_{\R^d} \langle \nabla f, \nabla g \rangle\, d\mu
-\int_{\R^d} \langle B, \nabla f \rangle g(x)\, d\mu ,\  f,g\in \{ f\in D(L^0)_{0,b}\,:\,\langle B,\nabla f\rangle \in L^2(E,\mu)\}.
$$
We consider the following condition on $B$:\\
\centerline{}
There exist constants $M>0$ and $\alpha\in \R$ such that
$$
|\langle B(x),x \rangle|\leq M(1+|x|)^\alpha
$$
for $\mu$-a.e. sufficiently large $|x|$. \\
\centerline{}
(a) Let $\varphi$ be a Muckenhoupt $\cal{A}$$_1$-weight and $d=2$. By  \cite[Proposition 1.2.7]{Tu}, for $r>1$ and some constant $A$
$$
v_1(r)\leq A r^2.
$$
Since $\varphi\in \cal{A}$$_1$, there exists $p>1$ such that $\varphi^p\in \cal{A}$$_1$ (cf. \cite[IX. Theorem 3.5 (Reverse H\"older)]{Ta}). We may assume that $\alpha\cdot \frac{p}{p-1}+2\neq 0$ (otherwise choose a slightly bigger $p$). Note that for sufficiently large $r>1$,
\begin{align*}
v_2(r)=\int_{B_r}|\langle B(x),x \rangle |\varphi(x)dx &\leq \Big (\int_{B_r} \varphi^p dx\Big )^\frac{1}{p} \Big (\int_{B_r} |\langle B(x),x \rangle | ^\frac{p}{p-1} dx\Big )^\frac{p-1}{p}\\
&\displaystyle\leq C (1+r)^{\alpha +\frac{2}{p}+\frac{2p-2}{p}}=C(1+r)^{\alpha +2},
\end{align*}
where $C$ is some positive constant. By Corollary \ref{result}, if $\alpha\leq -2$, then $\E$ is not transient.\\
\centerline{}
(b) Let $\varphi$ be a Muckenhoupt $\cal{A}$$_\beta$-weight with $1\leq \beta \leq 2$. Then the assumptions (A), (B) and (C) in \cite{Stu2} are satisfied on $\R^d$ for $(\E^0,D(\E^0))$ (cf. \cite[Lemma 5.2]{RST}). Furthermore, by \cite[Proposition 2.3]{Stu2} and \cite[Section 4]{Stu3}, there exists a measurable function $(p^0_t(x,y))_{t>0,x,y\in \R^d}$ and some constant $C>0$ depending on $\beta$, $d$ and $A$ such that
$$
T^0_t f(x)=\int_{\R^d} p^0_t(x,y) f(y) \mu (dy)\ \mu\text{-a.e.}\ x\in \R^d,
$$
where $f\in L^2(\R^d,\mu)$ and ($T^0_t)_{t>0}$ denotes the $C_0$-semigroup of contractions on $L^2(\R^d,\mu)$ corresponding to $(\E^0,D(\E^0))$ and for any $t>0$  $x,y\in \R^d$,
$$
\ \frac{1}{C\cdot \mu (B_{\sqrt{t}} (y)) } \exp\Big (- \frac{C |x-y|^2}{t} \Big)\leq p^0_t(x,y)\leq \frac{C(1+\frac{|x-y|^2}{t})^{\frac{\beta d+\log_2 A}{2}}}{\sqrt{\mu (B_{\sqrt{t}} (x))\mu (B_{\sqrt{t}} (y))}}\exp \Big(-\frac{|x-y|^2}{4t}\Big ).
$$
By Remark \ref{rec2}(c), $(\E^0,D(\E^0))$ is irreducible. Consequently, by \cite[Corollary 4.12]{Stu3}, 
Remark \ref{rec2}(b) and (d)
\begin{equation}\label{transiencesym}
\int_1^\infty \frac{r}{v_1(r)}dr<\infty,
\end{equation}
if and only if $(\E^0,D(\E^0))$ is transient. Hence, by Corollary \ref{cor}(a),  (\ref{transiencesym}) is a sufficient 
criterion for transience of $\E$.\\
\centerline{}
(c) Let $\varphi(x):=|x|^\eta$ with $-d<\eta$. Note that then $C_0^{\infty}(\R^d\setminus\{0\})\subset D(L^0)_{0,b}$ for any $\eta>-d$, hence (C) holds. Then
$$
v_1(r)=\int_{B_r} |x|^\eta dx=C_1 r^{d+\eta},
$$
where $C_1$ depends on $d$ and for sufficiently large $r>1$,
$$
v_2(r)=\int_{B_r}|\langle B(x),x \rangle|\cdot |x|^\eta dx \leq \begin{cases}\ C_2(1+r)^{\alpha+d+\eta} \qquad &\  \alpha+d+\eta\neq 0,\\
\ C_2 \log{(1+r)} \qquad &\ \alpha+d+\eta= 0,
\end{cases}
$$
where $C_2$ depends on $d$ and $M$. For $-d<\eta<d$, it is well-known that $\varphi \in \cal{A}$$_2$. By (b) (cf. (\ref{transiencesym})), if $-d+2<\eta<d$, then $\E$ is transient. Moreover by Corollary \ref{result}, if one of the following conditions is satisfied, then $\E$ is not transient.
\begin{itemize}
	\item[(c1)] $d+\eta=2$ and $\alpha\leq -2$.
	\item[(c2)] $d+\eta\in(0,2)$ and $\alpha+d +\eta<2$.
\end{itemize}
Similarly to \cite[Section 3]{St1} one can show that there exists a diffusion process associated with $\E$ and similarly to \cite[Theorem 4.5]{T2} one can then derive a semimartingale characterization of this process. In particular, if $d+\eta\in(0,1]$, then the associated process will not be semimartingale. Thus (c2) asserts that we are able to determine non-transience or recurrence of this process even in the non-semimartingale case.
\section{Proofs of Lemmas \ref{prop}, \ref{con1} and \ref{con2}}\label{4}
\begin{proof}(of Lemma \ref{prop}) Let $V\subset \subset E$. Then $(\E^0,D(\E^{0,V}))$ is a regular sectorial Dirichlet form on $L^2(V,\mu)$. Denote by $(L^{0,V},D(L^{0,V}))$ its generator on $L^2(V,\mu)$. The following results can be derived similarly to in \cite[Proposition 1.1, Theorem 1.5 and Lemma 1.6]{St1}. We obtain:
\begin{itemize}
	\item[(a)] $L^V u:= L^{0,V}u+ \langle B, \nabla u \rangle$, $u\in D(L^{0,V})_b$ is 		closable on $L^1(V,\mu)$. The closure $(\overline{L}^V,D(\overline{L}^V))$ generates a sub-Markovian $C_0$-resolvent of contractions $(\overline{G}^V_\alpha )_{\alpha>0}$.\	\
	\item[(b)] $D(\overline{L}^V)_b\subset D(\E^{0,V})$ and for $u\in D(\overline{L}^V)_b$, $v\in D(\E^{0,V})_b$, we have\\
$$
\E^0(u,v)-\int_V \langle B,\nabla u \rangle vd\mu=-\int_V \overline{L}^Vu  v d\mu
$$
and
$$
\E^0(u,u)=-\int_V \overline{L}^Vu  u d\mu.
$$
\end{itemize}
Define for $f\in L^1(E,\mu)$, $\alpha>0$
$$
\overline{G}^V_\alpha f:= \overline{G}^V_\alpha (f \cdot 1_V).
$$
Then $(\overline{G}^V_\alpha)_{\alpha>0}$ can be extended to a sub-Markovian $C_0$-resolvent of contractions on $L^1(E,\mu)$. Indeed, if $V_1$ and $V_2$ are relatively compact open sets in $E$ and $V_1 \subset V_2$, then for $f\in  L^1(E,\mu) $ with $f\geq 0$ $\mu$-a.e. it holds (cf. \cite[Lemma 1.6]{St1})\\
\begin{equation}\label{gv1}
\overline{G}_\alpha ^{V_1} f \leq \overline{G}_\alpha^{V_2} f,\ \mu\text{-a.e.}
\end{equation}
Let $(V_n)_{n \geq 1}$ be relatively compact open sets in $E$ such that $\overline{V}_n \subset V_{n+1}$ for all $n\geq 1$ and $\cup_{n\geq 1} V_n=E$. If $f\in L^1(E,\mu)$ with $f \geq 0$ $\mu$-a.e., then
$$
\overline{G}_\alpha f:=\lim_{n\rightarrow \infty} \overline{G}_\alpha ^{V_n} f
$$
exists $\mu$-a.e. and is independent of the choice of relatively open sets $(V_n)_{n \geq 1}$ by (\ref{gv1}). For general $f\in L^1(E,\mu)$, let $\overline{G}_\alpha f:= \overline{G}_\alpha f^+ - \overline{G}_\alpha f^-$. Then $(\overline{G}_\alpha)_{\alpha>0}$ is a sub-Markovian $C_0$-resolvent of contractions on $L^1(E,\mu)$ (cf. \cite[Theorem 1.5(a)]{St1}). If $(\overline{L},D(\overline{L}))$ is the generator of  $(\overline{G}_\alpha)_{\alpha>0}$, then it satisfies conditions of Lemma \ref{prop} (see, \cite[Theorem 1.5]{St1}).
\end{proof}
\centerline{}
Note that $(\overline{L},D(\overline{L}))$ is a closed extension of $(L,D(L^0)_{0,b})$ on $L^1(E,\mu)$, but not necessarily the closure.\\
\centerline{}
\begin{proof}(of Lemma \ref{con1}) Let $V\subset \subset E$. Since $\E_1^0$- and $\E_1^{0,h}$-norms are equivalent on $D(\E^{0,V})$, $(\E^{0,h},D(\E^{0,V}))$ is also a regular sectorial Dirichlet form on $L^2(V,\mu)$. Denote by $(L^{0,h,V},D(L^{0,h,V}))$ its generator on $L^2(V,\mu)$. Then we obtain $D(L^{0,V})=D(L^{0,h,V})$ and $L^{0,h,V}u=L^{0,V}u-h\cdot u$ for $u\in D(L^{0,V})=D(L^{0,h,V})$. Furthermore, similarly as in Lemma \ref{prop}, we obtain:
\begin{itemize}
	\item[(c)] $L^{h,V} u:= L^{0,h,V}u+ \langle B, \nabla u \rangle$, $u\in D(L^{0,h,V})_b$ is closable on $L^1(V,\mu)$. The closure $(\overline{L}^{h,V},D(\overline{L}^{h,V}))$ generates a sub-Markovian $C_0$-resolvent of contractions $(\overline{G}^{h,V}_\alpha )_{\alpha>0}$.
	\item[(d)] $D(\overline{L}^{h,V})_b\subset D(\E^{0,V})$ and for $u\in D(\overline{L}^{h,V})_b$, $v\in D(\E^{0,V})_b$, we have\\
$$
\E^{0,h}(u,v)-\int_V \langle B,\nabla u \rangle vd\mu=-\int_V \overline{L}^{h,V} u  v d\mu
$$
and\\
$$
\E^{0,h}(u,u)=-\int_V \overline{L}^{h,V}u  u d\mu.
$$
	\item[(e)] $D(\overline{L}^{h,V})=D(\overline{L}^V)$ and for $u\in D(\overline{L}^{h,V})$,
$$
\overline{L}^{h,V} u= \overline{L}^Vu-h\cdot u.
$$
\end{itemize}
Since the graph norms of $L^{h,V}$ and $L^V$ are equivalent on $D(L^{0,V})$, we obtain the last statement (e).\\
Define for $f\in L^1(E,\mu)$, $\alpha>0$
$$
\overline{G}^{h,V}_\alpha f:= \overline{G}^{h,V}_\alpha (f \cdot 1_V).
$$
Then similarly to the above $(\overline{G}^{h,V}_\alpha)_{\alpha>0}$ can be extended to a sub-Markovian $C_0$-resolvent of contractions on $L^1(E,\mu)$. As in the $(\overline{G}_\alpha)_{\alpha>0}$ case, choose relatively compact open sets $(V_n)_{n \geq 1}$ such that $\overline{V}_n \subset V_{n+1}$ for all $n\geq 1$ and $\cup_{n\geq 1} V_n=E$. Then for $f\in L^1(E,\mu)$ with $f\geq 0$ $\mu$-a.e. $\overline{G}_\alpha^{h,V_n} f$ is increasing in $n$ and
$$
\overline{G}^h_\alpha f:=\lim_{n\rightarrow \infty} \overline{G}_\alpha ^{h,V_n} f
$$
exists $\mu$-a.e. and is independent of the choice of relatively compact open sets $(V_n)_{n \geq 1}$. For a general $f\in L^1(E,\mu)$, let $\overline{G}^h_\alpha f:= \overline{G}^h_\alpha f^+ - \overline{G}^h_\alpha f^-$. Then $(\overline{G}_\alpha^h)_{\alpha>0}$ is a sub-Markovian $C_0$-resolvent of contractions on $L^1(E,\mu)$ and its generator $(\overline{L}^h,D(\overline{L}^h))$ satisfies properties (a) and (b) of Lemma \ref{con1}.\\
Next, we show that $D(\overline{L}^h)=D(\overline{L}).$ By definition, if $u\in D(\overline{L})$, then there exists $f\in  L^1(E,\mu)$ such that
$$
u=\overline{G}_\alpha f= \lim_{n\rightarrow \infty} \overline{G}_\alpha ^{V_n} (f\cdot 1_{V_n})
$$
where $V_n \subset\subset E$, $\overline{V}_n\subset V_{n+1}$ and $V_n \nearrow E$. Without loss of generality, we may assume that $f\geq 0$. Since $V_n$ is relatively compact open for $n\geq 1$, $D(\overline{L}^{V_n})=D(\overline{L}^{h,V_n})$. So there exists a sequence of functions $(g_n)_{n\geq 1}\subset L^1(E,\mu)$ with supp($g_n)\subset \overline V_n$ such that
$$
\overline{G}_\alpha^{V_n} (f\cdot 1_{V_n})=\overline{G}_\alpha^{h,V_n}g_n
$$
for all $n\geq 1$. In particular, $(\alpha -\overline{L}^{h,V_n}) \overline{G}_\alpha^{V_n} (f\cdot 1_{V_n})=g_n$. Since $-\overline{L}^{h,V_n}u= -\overline{L}^{V_n} u +h\cdot u$ for any $u\in D(\overline{L}^{h,V_n})=D(\overline{L}^{V_n})$,
$$
g_n=f \cdot 1_{V_n} +h\cdot \overline{G}_\alpha ^{V_n}(f\cdot 1_{V_n})\geq 0.
$$
Since $\overline{G}_\alpha ^{V_n} (f \cdot 1_{V_n})$ is increasing in $n$ and converges to $\overline{G}_\alpha f$ $\mu$-a.e. and in $L^1(E,\mu)$, g:=$\lim_{n \rightarrow \infty} g_n=f+h\cdot \overline{G}_\alpha f$ exists $\mu$-a.e.\\
We claim that $\displaystyle\lim_{n \rightarrow \infty} \overline{G}_\alpha^{h,V_n}(g\cdot 1_{V_n})=\lim_{n \rightarrow \infty}\overline{G}_\alpha ^{h,V_n} g_n$. Since $\overline{G}_\alpha^{V_n}(f 1_{V_n})\leq \overline{G}_\alpha f$, we have $g 1_{V_n}\geq g_n$, hence $\overline{G}_\alpha^{h,V_n}(g\cdot 1_{V_n})\geq \overline{G}_\alpha^{h,V_n} g_n$ and
\begin{align*}
\|\overline{G}_\alpha^{h,V_n}(g\cdot 1_{V_n})-\overline{G}_\alpha ^{h,V_n}g_n\|_{L^1(\mu)} &=\| \overline{G}_\alpha ^{h,V_n}\Big[ 1_{V_n}h \Big(\overline{G}_\alpha f -\overline{G}_\alpha^{V_n}(f \cdot 1_{V_n})\Big) \Big] \|_{L^1(\mu)} \\
&\leq \frac{1}{\alpha} \|h\|_{L^\infty(\mu)} \| \overline{G}_\alpha f- \overline{G}_\alpha ^{V_n}(f \cdot 1_{V_n})\|_{L^1(\mu)}.
\end{align*}
Since $\lim_{n \rightarrow \infty} \overline{G}_\alpha ^{V_n} (f \cdot 1_{V_n})=\overline{G}_\alpha f $ in $L^1 (E,\mu)$ and $\lim_{n \rightarrow \infty} \overline{G}_\alpha ^{h,V_n}(g \cdot 1_{V_n})=\overline{G}_\alpha^h g$ in $L^1(E,\mu)$, we have $\lim_{n \rightarrow \infty} \overline{G}_\alpha^{h,V_n} g_n = \overline{G}_\alpha^h g $ in $L^1(E,\mu)$. Since $\overline{G}_\alpha^{h,V_n} g_n$ is increasing in $n$, it converges $\mu$-a.e. hence, $\lim_{n\rightarrow \infty} \overline{G}_\alpha^{h,V_n} g_n = \overline{G}_\alpha ^h g $ $\mu$-a.e. and in $L^1(E,\mu)$. Therefore
$$
\overline{G}_\alpha f=\lim_{n \rightarrow \infty} \overline{G}_\alpha ^{h,V_n}g_n =\lim_{n \rightarrow \infty} \overline{G}_\alpha ^{h,V_n}(g \cdot 1_{V_n}) =\overline{G}_\alpha^h g
$$
with $g=f+h\cdot \overline{G}_\alpha f$ and so $D(\overline{L}) \subset D(\overline{L}^h)$.\\
Likewise, if $u\in D(\overline{L}^h)$ such that $u=\overline{G}_\alpha^h f=\lim_{n \rightarrow \infty} \overline{G}_\alpha^{h,V_n}(f \cdot 1_{V_n})$ where $f\in L^1(E,\mu)$ with $f\geq 0$, then $u\in D(\overline{L})$ and $\overline{G}_\alpha^h f=\overline{G}_\alpha g$, where $g=f-h\cdot \overline{G}_\alpha^h f.$
\end{proof}
\centerline{}
\begin{proof}(of Lemma \ref{con2}) Consider the real Hilbert space $L^2(E,(h+\e)\mu)$. Then it is easy to see that $(\E^{0,V},D(\E^{0,V}))$ is a regular sectorial Dirichlet form on $L^2(V,(h+\e)\mu)$. Denote by $(L^{0,\e,V},D(L^{0,\e,V}))$ the $L^2(V,(h+\varepsilon)\mu)$-generator of $(\E^{0,V},D(\E^{0,V}))$ on $L^2(V,(h+\e)\mu)$. Then we can show $D(L^{0,\e,V})=D(L^{0,V})$ as before. Furthermore, we obtain:
\begin{itemize}
	\item[(f)] $L^{\e,V}u:=L^{0,\e,V}u+\langle B^\e,\nabla u \rangle,$ $u\in D(L^{0,\e,V})_b$ is closable on $L^1(V,(h+\e)\mu).$ The closure $(\overline{L}^{\e,V},D(\overline{L}^{\e,V}))$ generates a sub-Markovian $C_0$-resolvent of contractions $(\overline{G}_\alpha^{\e,V})_{\alpha > 0}$.
	\item[(g)] $D(\overline{L}^{\e,V})_b \subset D(\E^{0,\e,V})$ and for $u\in D(\overline{L}^{\e,V})_b$, $v\in D(\E^{0,\e,V})_b$, we have\\
$$
\E^0(u,v)-\int_V \langle B^\e,\nabla u \rangle v (h+\e)d\mu=-\int_V \overline{L}^{\e,V} uv (h+\e)d\mu
$$
and
$$
\E^0(u,u)=-\int_V \overline{L}^{\e,V}uu (h+\e)d\mu.
$$
	\item[(h)] $D(\overline{L}^{\e,V})=D(\overline{L}^V)$ and for $u\in D(\overline{L}^{\e,V})$,
$$
\overline{L}^{\e,V}u=\frac{1}{h+\e}\overline{L}^V u.
$$
\end{itemize}
If we define for $f\in L^1(E,(h+\e)\mu),$ $\alpha>0$,
$$
\overline{G}_\alpha^{\e,V}f:=\overline{G}_\alpha^{\e,V}(f \cdot 1_{V}),
$$
then $(\overline{G}_\alpha^{\e,V})_{\alpha>0}$ can be extended to a sub-Markovian $C_0$-resolvent of contractions on $L^1(E,(h+\e)\mu)$. As in the $(\overline{G}_\alpha)_{\alpha>0}$ case, choose relatively compact open sets $(V_n)_{n \geq 1}$ such that $\overline{V}_n \subset V_{n+1}$ for all $n\geq 1$ and $\cup_{n\geq 1} V_n=E$. Then for $f\in L^1(E,(h+\e)\mu)$ with $f\geq 0$ $\mu$-a.e., $\overline{G}_\alpha^{\e,V_n} f$ is increasing in $n$ and
$$
\overline{G}^\e_\alpha f:=\lim_{n\rightarrow \infty} \overline{G}_\alpha ^{\e,V_n}f
$$
exists $\mu$-a.e. and is independent of the choice of relatively compact open sets $(V_n)_{n \geq 1}$. For a general $f\in L^1(E,(h+\e)\mu)$, let $\overline{G}^\e_\alpha f:= \overline{G}^\e_\alpha f^+ - \overline{G}^\e_\alpha f^-$. Then $(\overline{G}_\alpha^\e)_{\alpha>0}$ is a sub-Markovian $C_0$-resolvent of contractions on $L^1(E,(h+\e)\mu)$ and its generator $(\overline{L}^\e,D(\overline{L}^\e))$ satisfies properties (a) and (b) of Lemma \ref{con2}.
\centerline{}
Next we show $D(\overline{L}^\e)=D(\overline{L})$. If $u\in D(\overline{L}^\e)$, then there exists $f\in L^1(E,(h+\e)\mu)$ such that $u=\overline{G}_\alpha^\e f= \lim_{n\rightarrow \infty} \overline{G}_\alpha ^{\e,V_n}(f \cdot 1_{V_n})$ where $V_n \subset \subset E$ and $V_n \nearrow E$. Without loss of generality, we may assume that $f\geq 0$. Since $V_n$ is relatively compact open in $E$ for all $n\geq 1$, $D(\overline{L}^{\e,V_n})=D(\overline{L}^{V_n})$. So there exists a sequence of functions $(g_n)_{n\geq 1}\subset L^1(E,\mu)$ with supp$(g_n)\subset \overline V_n$ such that
$$
\overline{G}_\alpha^{\e,V_n}(f \cdot 1_{V_n})= \overline{G}_\alpha^{V_n} g_n
$$
for all $n\geq 1$. So we have $(\alpha - \overline{L}^{V_n})\overline{G}_\alpha^{\e,V_n}(f \cdot 1_{V_n})=g_n$. Since $\overline{L}^{\e,V_n}u=\frac{1}{h+\e}\overline{L}^{V_n}u$ for $u\in D(\overline{L}^{V_n})=D(\overline{L}^{\e,V_n})$, $g_n=(h+\e)f \cdot 1_{V_n} +\alpha(1-(h+\e))\overline{G}_\alpha^{\e,V_n}(f \cdot 1_{V_n})$. Since $\lim_{n\rightarrow \infty} \overline{G}_\alpha^{\e,V_n}(f \cdot 1_{V_n})=\overline{G}_\alpha^\e f$ $(h+\e)\mu$-a.e. and in $L^1(E,(h+\e)\mu)$, $g:=\lim_{n \rightarrow \infty}g_n =(h+\e)f+\alpha(1-(h+\e))\overline{G}_\alpha^\e f$ exists $\mu$-a.e.\\
We claim that $\lim_{n\rightarrow \infty} \overline{G}_\alpha^{V_n} g_n = \lim_{n\rightarrow \infty} \overline{G}_\alpha^{V_n}(g \cdot 1_{V_n})$.\\
Indeed,
\begin{align*}
\| \overline{G}_\alpha^{V_n}(g \cdot 1_{V_n})-\overline{G}_\alpha^{V_n}g_n \|_{L^1(\mu)} &\leq (1+ \|h\|_{L^\infty(\mu)}+\e)\cdot \| \overline{G}_\alpha^\e f-\overline{G}_\alpha^{\e,V_n}(f \cdot 1_{V_n})\| _{L^1(\mu)} \\
&\leq \frac{1}{\e}(1+ \|h\|_{L^\infty(\mu)}+\e) \cdot \| \overline{G}_\alpha^\e f-\overline{G}_\alpha^{\e,V_n}(f \cdot 1_{V_n})\| _{L^1((h+\e)\mu)},
\end{align*}
and since $\lim_{n \rightarrow \infty} \overline{G}_\alpha ^{\e,V_n}(f \cdot 1_{V_n})=\overline{G}_\alpha^\e f $ in $L^1(E,(h+\e)\mu)$ and $\lim_{n \rightarrow \infty} \overline{G}^{V_n}_\alpha (g \cdot 1_{V_n})=\overline{G}_\alpha g$ in $L^1(E,\mu)$, we have $\lim_{n \rightarrow \infty} \overline{G}^{V_n}_\alpha g _n=\overline{G}_\alpha g$ in $L^1(E,\mu)$. Since $\overline{G}_\alpha^{V_n} g_n $ is increasing in $n$, it converges $\mu$-a.e. moreover $\lim_{n \rightarrow \infty} \overline{G}^{V_n}_\alpha g_n=\overline{G}_\alpha g$ $\mu$-a.e. and in $L^1(E,\mu)$. Therefore,
$$
\overline{G}^\e_\alpha f=\lim_{n \rightarrow \infty}\overline{G}_\alpha^{\e,V_n}(f \cdot 1_{V_n}) =\lim_{n \rightarrow \infty}\overline{G}_\alpha^{V_n}(g \cdot 1_{V_n})= \overline{G}_\alpha g.
$$
Likewise, we can show converse that if $f\in D(\overline{L}),$ then there exists a function $g\in L^1(E,(h+\e)\mu)$ such that $\overline{G}_\alpha f= \overline{G}_\alpha^\e g.$
\end{proof}
\centerline{}
{\bf Acknowledgments.} The second named author would like to thank Wilhelm Stannat for leaving him his notes on recurrence for personal use several years ago. The research of Minjung Gim was supported by Global PH.D Fellowship Program through the National Research Foundation of Korea(NRF) funded by the Ministry of Education (2012-H1A2A1004352).
The research of Gerald Trutnau was supported by NRF-DFG Collaborative Research program and Basic Science Research Program through the National Research Foundation of Korea (NRF-2012K2A5A6047864 and NRF-2012R1A1A2006987) and by DFG through Grant Ro 1195/10-1.

\bigskip
\noindent
{\it Gerald Trutnau\\Department of Mathematical Sciences and\\ 
Research Institute of Mathematics of Seoul National University\\
599 Gwanak-Ro, Gwanak-Gu, Seoul 08826, South Korea\\
E-mail: trutnau@snu.ac.kr\\ \\
Minjung Gim \\
National Institute for Mathematical Sciences\\ 
70 Yuseong-daero 1689 beon-gil, Yuseong-gu, Daejeon 34047, South Korea\\
E-mail: mjgim@nims.re.kr}


\begin{thebibliography}{XXX}

\bibitem{Ba} Bhattacharya, R. N.: {\it Criteria for recurrence and existence of invariant measures for multidimensional diffusions}. Ann. Probab. 6, no. 4, 541-553. 1978.


\bibitem{BeCiR} Beznea, L., C\^impean, I., R\"ockner, M.:  {\it Irreducible recurrence, ergodicity, and extremality of invariant measures for resolvents}, arXiv:1409.6492.
     

\bibitem{FOT} Fukushima M., Oshima Y., Takeda M.: {\it Dirichlet forms and Symmetric Markov processes}. Berlin-New York: Walter de Gruyter. 2011.

\bibitem{F07} Fukushima M.: {\it Transience, Recurrence and Large Deviation of Markov Processes}.  Bielefeld IGK Seminar. 2007.

\bibitem{GT} Gim M., Trutnau G.: {\it Explicit recurrence criteria for symmetric gradient type Dirichlet forms 
satisfying a Hamza type condition}. Mathematical Reports. Volume 15(65). No.4. 2013.

\bibitem{Ge} Getoor R. K.: {\it Transience and recurrence of Markov processes}. Seminar on Probability. XIV. pp. 397--409. Lecture Notes in Math. 784, Springer. Berlin. 1980.

\bibitem{H} Khas'minskii R. Z.: {\it Ergodic properties of recurrent diffusion processes and stabilization of the solution of the Cauchy problem for parabolic equations}. Theory Probab. Appl. Volume 5, Issue 2. 179-196.

\bibitem{K} Kuwae K.: {\it Invariant sets and ergodic decomposition of local semi-Dirichlet forms.} Forum Mathematicum. Volume 23. Issue 6. Pages 1259-1279. 2010.

\bibitem{MR} Ma Z.,  R\"ockner M.: {\it Introduction to the Theory of (Non-Symmetric) Dirichlet Forms.} Berlin: Springer 1992.

\bibitem{WaMaU} Masamune, J., Uemura, T., Wang, J.:
{\it On the conservativeness and the recurrence of symmetric jump-diffusions}. J. Funct. Anal. 263, no. 12, 3984-4008. 2012.


\bibitem{MT} Meyn S. P., Tweedie R. L.: {\it Markov Chains and Stochastic Stability.} Springer-Verlag. London. 1993.

\bibitem{No} Norris J. R.: {\it Markov chains.}  Cambridge University Press. 1997.

\bibitem{O} Oshima Y.: {\it On conservativeness and recurrence criteria for Markov processes.} Potential Analysis. 1992. Volume 1. Issue 2. 115-131.
 
\bibitem{O04} Oshima Y.: {\it Time-dependent Dirichlet forms and related stochastic calculus.} Infin. Dimens. Anal. Quantum Probab. Relat. Top. 7. No. 2. 281-316.

\bibitem{O13} Oshima Y.: {\it Semi-Dirichlet Forms and Markov Processes.} Walter de Gruyter 2013.

\bibitem{ORT} Ouknine Y., Russo F., Trutnau G.: {\it On countably skewed Brownian motion with accumulation point.} 
Electronic J. of Probability {\bf 20}, no. 82, 1-27, 2015.


\bibitem{Pi}Pinsky R. G.: {\it Positive Harmonic Functions and Diffusion.} Cambridge Studies in Advanced Mathematics 45. Cambridge University Press. Cambridge. 1995.

\bibitem{DR} Da Prato G., R\"ockner M.: {\it Singular dissipative stochastic equations in Hilbert spaces.} Probability Theory and Related Fields October 2002. Volume 124. Issue 2. 261-303.

\bibitem{RST}  R\"ockner M., Shin J., Trutnau G.: {\it Non symmetric distorted Brownian motion: strong solutions, strong Feller property and non-explosion results.} Discrete Contin. Dyn. Syst. Ser. B 21 (2016), no. 9, 3219--3237.

\bibitem{Sc} Schilling R. L.: {\it A note on invariant sets.} Probability and Mathematical Statics. Volume 24. fasc 1. 2004.


\bibitem{Stu1} Sturm K. T.: {\it Analysis on local Dirichlet spaces. I. Recurrence, conservativeness and $L^p$- Liouville properties.} J. Reine Angew. Math. 456 1994. 173-196.

\bibitem{Stu2} Sturm K. T.: {\it Analysis on local Dirichlet spaces. II. Upper Gaussian estimates for the fundamental solutions of parabolic equations.} Osaka J. Math. Volume 32. Number 2 1995. 275-312.

\bibitem{Stu3} Sturm K. T.: {\it Analysis on local Dirichlet spaces. III. The parabolic Harnack inequality.} J. Math. Pures Appl. (9) 75. 1996. No. 3.  273-297.

\bibitem{St1} Stannat W.: {\it (Nonsymmetric) Dirichlet operators on $L^1$: Existence, uniqueness and associated Markov processes.} Ann. Scuola Norm. Sup. Pisa Cl. Sci. (4) 28. 1999. No. 1. 99-140.

\bibitem{St2} Stannat W.: {\it The Theory of Generalized Dirichlet Forms and Its Applications in Analysis and Stochastics.} Dissertation, Bielefeld 1996. Published as Memoirs of the AMS. Volume 142. No. 678. 1999.

\bibitem{Ta} Torchinsky A.: {\it Real-Variable Methods in Harmonic Analysis.} Courier Corporation 2012.

\bibitem{T2} Trutnau G.: {\it Skorokhod decomposition of reflected diffusions on bounded Lipschitz domains with singular non reflection part.} Probability Theory and Related Fields. Volume 127, Issue 4. 2003.


\bibitem{Tu} Turesson B. O.: {\it Nonlinear Potential Theory and Weighted Sobolev Spaces.} Lecture notes in Mathematics. 1736. Springer. 2000.


\end{thebibliography}
\end{document}